\documentclass[reqno,a4paper]{amsart}

\usepackage[latin1]{inputenc}
\usepackage[italian,english]{babel}
\usepackage{eucal,amsfonts,amssymb,amsmath,amsthm,epsfig,mathrsfs}
\usepackage{dsfont}
\usepackage{color}
\usepackage{amscd,amsxtra}
\usepackage{enumerate}
\usepackage{latexsym}
\usepackage{bm}

\allowdisplaybreaks

\newcounter{ipotesi}

\Alph{ipotesi}
 \makeatletter \@addtoreset{equation}{section}

\makeatother \makeatletter

\newtheorem{theorem}{Theorem}[section]
\newtheorem{hypotheses}[theorem]{Hypotheses}{\rm}
\newtheorem{hypothesis}[theorem]{Hypothesis}{\rm}
\newtheorem{hyps}[theorem]{Hypotheses}{\rm}
\newtheorem{hyps0}[theorem]{Hypotheses}{\rm}
\newtheorem{hyps00}[theorem]{Hypotheses}{\rm}
\newtheorem{hyps1}[theorem]{Hypothesis}{\rm}
\newtheorem{lemma}[theorem]{Lemma}
\newtheorem{corollary}[theorem]{Corollary}
\newtheorem{proposition}[theorem]{Proposition}

\newtheorem{remark}[theorem]{Remark}{\rm}
\newtheorem{example}[theorem]{Example}

\newcounter{parentenv}

\newcommand{\hs}[1]{\hskip -#1pt}

\newcommand{\R}{{\mathbb R}}

\newcommand{\N}{{\mathbb N}}

\newcommand{\Rd}{\mathbb R^d}

\newcommand{\A}{\bm{\mathcal A}}

\newcommand{\ve}{\varepsilon}
\newcommand{\s}{\sigma}
\newcommand{\si}{\sum_{i=1}^d}
\newcommand{\sj}{\sum_{j=1}^d}

\newcommand{\sij}{\sum_{i,j=1}^d}

\newcommand{\sijh}{\sum_{i,j,h=1}^d}

\newcommand{\sijhk}{\sum_{i,j,h,k=1}^d}

\newcommand{\shk}{\sum_{h,k=1}^d}

\newcommand{\sijhkl}{\sum_{i,j,h,k,l=1}^d}

\newcommand{\qq}{\qquad}
\newcommand{\q}{\quad}

\newcommand\ds{\displaystyle}

\begin{document}

\title[On the estimates of the derivatives of solutions]{On the estimates of the derivatives of solutions to nonautonomous Kolmogorov equations and their consequences}
\thanks{The authors are members of G.N.A.M.P.A. of the Italian Istituto Nazionale di Alta Matematica (INdAM). Work partially supported by the INdAM-GNAMPA Project 2016 ``Equazioni e sistemi di equazioni ellittiche e paraboliche associate ad operatori con coefficienti
illimitati e discontinui''.}
\author[L. Angiuli, L. Lorenzi]{Luciana Angiuli, Luca Lorenzi}
\address{L.A.: Dipartimento di Matematica e Fisica ``Ennio De Giorgi'', Universit\`a del Salento, Via per Arnesano, I-73100 LECCE (Italy)}
\address{L.L.: Dipartimento di Matematica e Informatica, Universit\`a degli Studi di Parma, Parco Area delle Scienze 53/A, I-43124 PARMA (Italy)}
\email{luciana.angiuli@unisalento.it}
\email{luca.lorenzi@unipr.it}
\maketitle
\vspace{1,5cm}
\begin{center}
\begin{minipage}[t]{10cm}

%----- first page ---------------- Abstract
\small{ \noindent \textbf{Abstract.}
We consider evolution operators $G(t,s)$ associated to a class of nonautonomous elliptic operators with
unbounded coefficients, in the space of bounded and continuous functions over $\Rd$. We prove some new pointwise estimates for the spatial derivatives of the function
$G(t,s)f$, when $f$ is bounded and continuous or much smoother. We then use these estimates to prove smoothing effects of the evolution operator in $L^p$-spaces.
Finally, we show how pointwise gradient estimates have been used in the literature to study the asymptotic behaviour of the evolution operator and to prove
summability improving results in the $L^p$-spaces related to the so-called tight evolution system of measures.
\medskip

%----- first page ------------- Keywords
\medskip

%----- first page -------------- AMS classification codes
\noindent \textbf{Mathematics~Subject~Classification~(2010):}
35K10, 35K15, 35B40.

}
\end{minipage}
\end{center}

%%if necessary :
%%\medskip
%%\tableofcontents
\bigskip

%----------------------------------------- SECTIONS AND SUBSECTION:
\section{Introduction}

In this paper we deal with nonautonomous
Kolmogorov elliptic operators
 formally defined on smooth functions $\psi:\Rd\to\R$ by
\begin{align}
(\A\psi)(t,x)=&\sum_{i,j=1}^dq_{ij}(t,x)D_{ij}\psi(x)+\sum_{j=1}^db_j(t,x)D_j\psi(x)+c(t,x)\psi(x)\notag\\
=&{\rm Tr}(Q(t,x)D^2\psi(x))+\langle b(t,x), \nabla
\psi(x)\rangle+c(t,x)\psi(x).
\label{oper-A}
\end{align}
for any $(t,x)\in I\times\Rd$, where $I$ is a right-halfline (possibly
$I=\R$).
In \cite{AngLor10Com, KunLorLun09Non} it has been proved that,
under mild assumptions on the possibly unbounded
coefficients $q_{ij}, b_i$ and $c$,
an evolution operator $(G(t,s))_{t \ge s\in I}$\footnote{In the rest of the paper, we will denote the evolution operator simply by $G(t,s)$.}
can be associated to the operator $\A$ in $C_b(\Rd)$:
for any $f\in C_b(\Rd)$ and $I\ni s<t$,
$G(t,s)f$ is the value at $t$ of the unique solution
$u\in C([s,+\infty)\times\Rd)\cap C^{1,2}((s,+\infty)\times\Rd)$
to the Cauchy problem
\begin{equation}
\left\{
\begin{array}{ll}
D_tu=\A u, & {\rm in}~(s,+\infty)\times\Rd,\\[1mm]
u(s,\cdot)=f, & {\rm in}~\Rd,
\end{array}
\right.
\label{dreams}
\end{equation}
which is bounded in each strip $[s,s+T]\times\Rd$.
In Section \ref{sect-2} we will show how a solution to problem \eqref{dreams} with the previous properties can
be obtained.

In recent years several properties of the family $G(t,s)$ have been
investigated in the space of bounded and continuous functions over $\Rd$.
Among all of them, uniform estimates for the derivatives of $G(t,s)f$ have
played an important role to prove
existence, uniqueness and optimal Schauder estimates for the classical
solution to the nonhomogeneous Cauchy problem
\begin{equation}
\left\{
\begin{array}{lll}
D_tu(t,x)=\A u(t,x)+g(t,x), & t\in [s,T], &x\in\R^d,\\[1.5mm]
u(s,x)=f(x), && x\in\R^d
\end{array}
\right. \label{nonom}
\end{equation}
for any $s\in I$ and $T>s$.
In \cite{Lorenzi-1}, global uniform estimates for the first-, second- and
third-order derivatives of $G(t,s)f$ have been proved under
 growth and dissipativity assumptions on the coefficients of $\A$. Note
that some dissipativity condition is necessary, as the one-dimensional
example in \cite[Example 5.1.12]{LorLibro-2} shows. Actually, to get existence and uniqueness of the
bounded solution to \eqref{nonom} which belongs to $C([s,T]\times\Rd)\cap C^{1,2}((s,T]\times\Rd)$, local uniform estimates for the derivatives of
$G(t,s)f$ are enough (see \cite{AL} where the semilinear equation
$D_tu=\A u+\psi(u)$ is considered).
Uniform gradient estimates have also been proved for the solution of the Cauchy
problem associated to $\A$ in $C_b(\Omega)$ ($\Omega\subset \Rd$
being unbounded with smooth boundary) with homogeneous non-tangential boundary
conditions, see \cite{AngLorNon}.

On the other hand, as in the classical case of bounded coefficients, it is
natural to extend each operator $G(t,s)$ to some $L^p$-space.
The autonomous case shows that the usual $L^p$-spaces related to the
Lebesgue measure are not the best choice as possible, see Example \ref{ex-PRS}. Sufficient conditions have been proved in
\cite{AngLor10Com} for $G(t,s)$ to preserve $L^p(\Rd)$ (see also \cite{AngLorPal16Lpe} for the case when the elliptic operator \eqref{oper-A} is replaced by
a system of elliptic operators), which, in particular, imply
rather strong growth assumptions on the coefficients of the operator $\A$.

The autonomous case shows that the right $L^p$-spaces where to study the
semigroup $T(t)$ (the autonomous counterpart of the evolution operator) are
those related to the so-called invariant measure $\mu$, a probability
measure, which exists under an additional algebraic assumption on the coefficients of
the operator $\A$ and it is characterized by the {\it invariance property}
\begin{eqnarray*}
\int_{\Rd}T(t)fd\mu=\int_{\Rd}fd\mu,\qquad\;\,f\in C_b(\Rd).
\end{eqnarray*}

In the nonautonomous case the situation is quite different. Indeed, the
invariant measure is replaced by a one parameter family of probability
measures $\{\mu_t: t\in I\}$,
usually referred to as {\it evolution system of measures}. Whenever such a
family exists, it is characterized by the property
\begin{equation}\label{invariance}
\int_{\Rd}G(t,s)f d\mu_t=\int_{\Rd}f d\mu_s,\qquad\;\,I\ni s<t,\;\,f\in
C_b(\Rd).
\end{equation}
Through this formula, each operator $G(t,s)$ can be extended to a
contraction mapping $L^p(\Rd,\mu_s)$ into $L^p(\Rd,\mu_t)$. The following
facts make the analysis
of the evolution operator in these $L^p$-spaces much more difficult, than
in the autonomous case:
\begin{itemize}
\item
evolution systems of measures are infinitely many in general and not explicit (except in some special case);
\item
for different values of $s$ and $t$, the measures $\mu_s$ and $\mu_t$
differ in general: even  if they are equivalent, since they are both
equivalent to the Lebesgue measure,
the spaces $L^p(\Rd,\mu_s)$ and $L^p(\Rd,\mu_t)$ are different.
\end{itemize}
Hence, to study the evolution operator $G(t,s)$ in this $L^p$-setting one cannot take
advantage of the classical results, which require to work in spaces
$L^p(J;X)$, where $J$ is an interval: in our situation $X$ depends on $t$!

Among the infinitely many evolution systems of measures, the ``more important'' ones are the tight evolution systems,
where tight means that for any $\varepsilon>0$ there exists $R_\varepsilon>0$
such that for any $t\in I$, $\mu_{t}(B_{R_\varepsilon})>1-\varepsilon$.
The family of tight evolution system of measures reduces to a unique system
when, for example, the evolution operator $G(t,s)$ satisfies the
pointwise gradient estimate
\begin{equation}\label{poi-es}
|\nabla_x (G(t,s)f)(x)| \le e^{\sigma(t-s)}(G(t,s)|\nabla f|)(x),\qquad
t>s,\;\, x \in \Rd
 \end{equation}
for any $f \in C^1_b(\Rd)$ and some constant $\sigma<0$. Estimate \eqref{poi-es} is
the key tool to prove
 a lot of important properties of the evolution operator $G(t,s)$
in the $L^p$-spaces related to the tight evolution system of
measures.

The aim of this paper is twofold. First, in Section \ref{sect-4} we prove different types of pointwise
estimates for the first-, second- and third-order spatial derivatives of $G(t,s)f$.
More precisely, we provide sufficient conditions for the estimates
\begin{align}
&|D^k_xG(t,s)f|^p\le \Gamma_{p,k}^{(1)}(t-s)G(t,s)\bigg (\sum_{j=1}^k|D^jf|^2\bigg )^{\frac{p}{2}},
\label{aa}
\\
&|D^k_xG(t,s)f|^p\le \Gamma_{p,h,k}^{(2)}(t-s) G(t,s)\bigg(\sum_{j=0}^{h}|D^j f|^2\bigg )^{\frac{p}{2}}
\label{aaaa}
\end{align}
to hold in $\Rd$ for any $t>s\in I$, $h\in\{0,\ldots,k\}$, $k=1,2,3$ and $p\in (p^*,+\infty)$ for a suitable $p^*\in [1,+\infty)$,
where $\Gamma_{p,k}^{(1)}$ and $\Gamma_{p,h,k}^{(2)}$ are positive functions.
All of these estimates are proved by using a variant of the maximum
principle for operator with unbounded coefficients and, as one expects,
 they are derived under more restrictive assumptions on the coefficients
of $\A$. We deal also with the case $p=1$ which is much more delicate and requires stronger assumptions.
Indeed, as \cite{Angiuli} shows, the algebraic condition
$D_l q_{ij}+D_i q_{lj}+D_j q_{il}=0$ in $I\times \Rd$ for any $i,j,l\in\{1,\ldots,d\}$ with
$i\neq l\neq j$, is a necessary condition for \eqref{aa} (with $k=1$) to hold.
For this reason many results are proved assuming that the diffusion
coefficients do not depend on the spatial variable.

Next in Section \ref{sect-5} we present many interesting consequences of the previous estimates in the study of
$G(t,s)$ in $L^p$-spaces. In particular, we stress the prominent role played by estimate \eqref{poi-es},
illustrating its main applications known in the literature.
First of all, estimate \eqref{poi-es} allows to prove a logarithmic Sobolev inequality for the unique tight evolution system of
measures $\mu_t$, i.e., the estimate
\begin{align}
\int_{\Rd} |f|^p\log|f|^p d\mu_s\leq &\|f\|_{L^p(\Rd,\mu_s)}^p\log(\|f\|_{L^p(\Rd,\mu_s)}^p)\notag\\
&+C_p\int_{\Rd}|f|^{p-2}|\nabla f|^2\chi_{\{f\neq 0\}}d\mu_s,
\label{Log_Sob-1}
\end{align}
for any smooth enough function $f$ and some positive constant $C_p$, independent of $f$.
Besides its own interest, which consists of the fact that \eqref{Log_Sob-1} is the counterpart of the Sobolev embeddings which fail in the $L^p$-spaces related to the measures $\mu_t$, see Example \ref{ex-noSob}, inequality \eqref{Log_Sob-1} is crucial to deduce the hypercontractivity of the operator $G(t,s)$ in the $L^p$-spaces related to $\mu_t$. Further and stronger summability improving properties of the operator $G(t,s)$ are also investigated and, in most the cases, a characterization of them is given in \cite{AngLorOnI}.

In the last subsection we deal with the time behaviour of $G(t,s)f$, as $t\to +\infty$, when $f \in L^p(\Rd, \mu_s)$.
Using the hypercontractivity of $G(t,s)$ and the Poincar\'e inequality in $L^2(\Rd, \mu_s)$, we connect the decay rate to zero of $\|G(t,s)f-\overline{f}_s\|_{L^p(\Rd, \mu_t)}$ to the decay rate to zero of $\|\nabla_x G(t,s)f\|_{L^p(\Rd, \mu_t)}$ as $t \to +\infty$, obtaining as a consequence an exponential decay rate to zero of $\|G(t,s)f-\overline{f}_s\|_{L^p(\Rd, \mu_t)}$ as $t \to +\infty$. Here, $\overline{f}_s$ denotes the average of $f$ with respect to $\mu_s$.
All these results are based heavily on the estimate \eqref{poi-es} whose validity, as already observed, is guaranteed under quite stronger assumptions on the coefficients of $\A$.
The convergence to zero of $\|G(t,s)f-\overline{f}_s\|_{L^p(\Rd, \mu_t)}$ can also be proved without assuming the validity of gradient estimates of negative type, using different argument that we present
with some details. As a matter of fact, in this situation, we can not prove an exponential decay rate.

We point out that the convergence results are quite involved since the measures $\mu_t$ depend themselves explicitly on time too.

\paragraph*{Notations}
Throughout the paper we use the subscripts ``$b$'' and ``$c$'', which stand for ``bounded'' and ``compactly supported''.
For instance, $C_b(\Rd)$ denotes the set of all bounded and continuous functions
$f:\Rd\to\R$. We endow it with the sup-norm $\|\cdot\|_{\infty}$.
For any $k>0$ (possibly $k=+\infty$), $C^k_b(\Rd)$ denotes the
subset of $C_b(\Rd)$ of all functions $f:\Rd\to\R$ that are
continuously differentiable in $\Rd$ up to $[k]$th-order, with
bounded derivatives and such that the $[k]$th-order derivatives are
$(k-[k])$-H\"older continuous in $\Rd$. $C^k_b(\Rd)$ is endowed
with the norm $\|f\|_{C^k_b(\Rd)}:=\sum_{|\alpha|\le
[k]}\|D^{\alpha}f\|_{\infty}
+\sum_{|\alpha|=[k]}[D^{\alpha}f]_{C^{k-[k]}_b(\Rd)}$.
For any domain $D\subset\R^{d+1}$ and $\alpha\in (0,1)$,
$C^{\alpha/2,\alpha}(D)$ denotes the space of all
H\"older-continuous functions with respect to the parabolic distance
of $\R^{d+1}$. Similarly, for any $h,k\in\mathbb N\cup\{0\}$ and
$\alpha\in [0,1)$, $C^{h+\alpha/2,k+\alpha}(D)$ denotes the set of
all functions $f:D\to\R$ which (i) are continuously differentiable
in $D$ up to the $h$th-order with respect to time variable, and up
to the $k$th-order with respect to the spatial variables, (ii) the
derivatives of maximum order are in $C^{\alpha/2,\alpha}(D)$ (here,
$C^{0,0}:=C$). By
$C^{h+\alpha/2,k+\alpha}_{\rm loc}(D)$ we denote the set of all
functions $f:D\to\R$ which are in $C^{h+\alpha/2,k+\alpha}(D_0)$ for
any compact set $D_0\subset D$.
For any measure positive $\mu$, the Sobolev space $W^{k,p}(\Rd,\mu)$ ($k\in\N\cup\{0\}$ $p\in [1,+\infty]$)
is the set of all functions $f\in L^p(\Rd,\mu)$, whose distributional derivatives up
to the $k$-th-order are in $L^p(\Rd,\mu)$. It is normed by setting
$\|f\|_{W^{k,p}(\Rd,\mu)}=\sum_{j=0}^k\|D^jf\|_{L^p(\Rd,\mu)}$. When $\mu$ is the Lebesgue measure
we simply write $W^{k,p}(\Rd)$.
For any real function $f$ we denote by $f^+$ and $f^-$ respectively its positive and negative part.
Finally, by $B_r$ and $\mathds{1}$ we denote, respectively, the open ball in $\Rd$ centered at
the origin with radius $r$  and
the function identically equal to one in $\Rd$.

\section{Main assumptions and preliminaries}
\label{sect-2}

Throughout the paper, we assume the following conditions on the coefficients of the operator $\A$ in \eqref{oper-A}.
\begin{hypotheses}
\label{hyp-1}
\begin{enumerate}[\rm (i)]
\item
$q_{ij}, b_i, c$ belong to $C^{\alpha/2,\alpha}_{\rm loc}(I\times \Rd)$ for some $\alpha \in (0,1)$ and any $i,j=1, \ldots,d$;
\item
the matrix $Q(t,x)$ is symmetric for any $(t,x)\in I\times \Rd$ and $\nu_0:=\inf_{I\times \Rd}\nu>0$ where
$\nu(t,x)$ is the minimum of the eigenvalues of $Q(t,x)$;
\item
$c_0:=\sup_{I\times \Rd}c<+\infty$;
\item
there exist a positive function $\varphi:\Rd\to\R$ blowing up as $|x|$ tends to $+\infty$, and, for any $[a,b]\subset I$, a positive constant $\lambda_{a,b}$ such that $\A\varphi\le\lambda_{a,b}\varphi$
in $[a,b]\times\Rd$.
\end{enumerate}
\end{hypotheses}

Under the previous set of assumptions in \cite{AngLor10Com,KunLorLun09Non} it has been proved that, for any $f\in C_b(\Rd)$ and $s\in I$, the Cauchy problem
\eqref{dreams} admits a unique solution $u\in C([s,+\infty)\times\Rd)\cap C^{1,2}((s,+\infty)\times\Rd)$ (a so-called {\it classical solution}) which is bounded in the strip $[s,T]\times\Rd$
for any $T>s$. In addition, $u$ satisfies the estimate
\begin{equation}
\|u(t,\cdot)\|_{\infty}\le e^{c_0(t-s)}\|f\|_{\infty},\qquad\;\,t>s.
\label{reality}
\end{equation}

Actually, the existence of a solution to problem \eqref{dreams} can be proved also without Hypothesis \ref{hyp-1}(iv) (which is used to prove the uniqueness of the solution) as
the following lemma shows.

\begin{lemma}
\label{lemma-2.1}
Under Hypotheses $\ref{hyp-1}(i)$-$(iii)$, for any $f\in C_b(\Rd)$ the Cauchy problem \eqref{dreams} admits a solution $u\in C([s,+\infty)\times\Rd)\cap C^{1,2}((s,+\infty)\times\Rd)$, which satisfies estimate
\eqref{reality}.
\end{lemma}

\begin{proof}
For any $n\in\N$ and any nonnegative function $f\in C_b(\Rd)$, consider the Cauchy-Dirichlet problem
\begin{equation}
\left\{
\begin{array}{ll}
D_tu=\A u, & {\rm in}~(s,+\infty)\times B_n,\\
u=0, & {\rm on}~(s,+\infty)\times\partial B_n,\\
u(s,\cdot)=f, & {\rm in}~B_n.
\end{array}
\right.
\label{dreams-1}
\end{equation}
It is well known that, for any $n\in\N$, the previous Cauchy problem admits a unique classical solution $u_n$ and it satisfies the estimate
\begin{equation}
\|u_n(t,\cdot)\|_{C(\overline{B_n})}\le e^{c_0(t-s)}\|f\|_{\infty},\qquad\;\,t>s.
\label{abba}
\end{equation}

As it is immediately seen, the function $w_n=u_n-u_{n+1}$ satisfies the inequality
$D_tw_n-\A w_n=0$ in $(s,+\infty)\times B_n$, is nonpositive on $(s,+\infty)\times\partial B_n$ and
vanishes on $\{s\}\times B_n$. The classical maximum principle shows that
$u_n\le u_{n+1}$ in $(s,+\infty)\times B_n$. Hence, for any $(t,x)\in (s,+\infty)\times\Rd$,
the sequence $(u_n(t,x))$ converges. We can thus define a function $u:(s,+\infty)\times\Rd\to\R$
by setting $u(t,x)=\lim_{n\to +\infty}u_n(t,x)$ for any $(t,x)\in (s,+\infty)\times\Rd$. Clearly, $u$
satisfies \eqref{reality}.
On the other hand, the convergence is also in $C^{1,2}([a,b]\times K)$ for any pair of compact sets $[a,b]\subset (s,+\infty)$ and $K\subset\Rd$ as
a consequence of the classical interior Schauder estimates (see e.g., \cite{LadSolUra68Lin})\footnote{Actually, the interior Schauder estimates imply via a compactness argument
that a subsequence $(u_{n_k})$ converges in $C^{1,2}([a,b]\times K)$ and the pointwise convergence of $(u_n)$ shows that, in fact, all the sequence
$u_n$ converges in $C^{1,2}([a,b]\times K)$.}. This implies that $u$ solves the differential equation in \eqref{dreams}.

Let us prove that $u$ can be extended to $[s,+\infty)\times\Rd$ with a continuous function and $u(s,\cdot)=f$.
We use a localization argument and, to avoid cumbersome notation, we denote simply by $C$ a positive constant, which may depend on $m$ but is independent of $k$, and may vary from line to line.
 We fix $m\in\N$ and a smooth function  $\vartheta$ such that $\chi_{B_m}\le\vartheta\le\chi_{B_{m+1}}$.
If $k>m$ then the function $v_k=\vartheta u_k$ belongs to $C([s,+\infty)\times \overline{B}_{m+1})$ and solves the Cauchy-Dirichlet problem
\begin{eqnarray*}
\left\{
\begin{array}{ll}
D_t v_k=\A v_k-\psi_k,&{\rm in}~(s,+\infty)\times B_{m+1},\\[1mm]
v_k=0,       & {\rm on}~(s,+\infty)\times\partial B_{m+1},\\[1mm]
v_k(s,\cdot)=\vartheta f, &{\rm in}~B_{m+1},
\end{array}\right.
\end{eqnarray*}
where $\psi_k=2\langle Q\nabla_xu_k,\nabla\vartheta\rangle+u_k{\rm Tr}(QD_{ij}\vartheta)+u_k\langle b,\nabla\vartheta\rangle$.
The solution to the previous nonhomogeneous Cauchy problem is given by the variation-of-constants formula
\begin{eqnarray*}
v_k(t,\cdot)=G_{m+1}(t,s)f-\int_{s}^tG_{m+1}(r,s)\psi_k(r,\cdot)dr,\qquad\;\,t>s,
\end{eqnarray*}
where $G_{m+1}$ denotes the evolution operator associated with the realization in $C_b(B_{m+1})$ of the operator $\A$ with homogenous Dirichlet boundary conditions.

Since $v_k=u_k$ in $(s,+\infty)\times B_m$, if $x$ belongs to $B_m$ then it holds that
\begin{equation}
|u_k(t,x)-f(x)|\le |(G_{m+1}(t,s)f)(x)-f(x)|+\int_{s}^t|(G_{m+1}(r,s)\psi_k(r,\cdot))(x)|dr,
\label{sfioravo}
\end{equation}
which implies that
\begin{align*}
|u(t,x)-f(x)|\le &|(G_{m+1}(t,s)f)(x)-f(x)|\\
&+\limsup_{k\to +\infty}\int_{s}^t|(G_{m+1}(r,s)\psi_k(r,\cdot))(x)|dr.
\end{align*}
Clearly, $(G_{m+1}(t,s)f)(x)$ converges to $f(x)$ as $t\to s^+$. On the other hand, the integral term vanishes as $t\to s^+$,
uniformly with respect to $k$. Indeed,
using \eqref{abba} we can straightforwardly estimate
\begin{equation}
\label{seizo-int-3}
|\psi_k(t,x)|\leq C(e^{c_0(t-s)}\|f\|_\infty+\|\nabla_xu_k(t,\cdot)\|_{L^\infty(B_{m+1})}),\qquad\;\,t>s,\;\,x\in B_{m+1}.
\end{equation}
Moreover, the estimates in \cite[Theorem 3.5]{friedman} and \eqref{ariosto} show that
$|\nabla_xu_k(t,x)|\le C(t-s)^{-1/2}\|f\|_{\infty}$
for any $t\in (s,s+1]$ and $x\in B_{m+1}$.
Combining this estimate and \eqref{seizo-int-3} we deduce that
$|\psi_k(t,x)|\leq
C(t-s)^{-1/2}\|f\|_\infty$ for any $(t,x)\in (s,s+1]\times B_{m+1}$,
and $k>m$. Since $\|G_{m+1}(t,s)\|_{L(C_b(\Rd))}\le e^{c_0(t-s)}$ for any $t>s$ and $m\in\N$, it thus follows that
$|(G_{m+1}(r,s)\psi_k(r,\cdot))(x)|\le C(r-s)^{-1/2}$ for any $(r,x)\in (s,s+1]\times B_{m+1}$, and it is now clear that the integral term in
the right-hand side of \eqref{sfioravo} vanishes as $t\to s^+$, uniformly with respect to $k$.

By the arbitrariness of $m$ we have so proved the assertion of the theorem for nonnegative functions $f\in C_b(\Rd)$.

For a general $f\in C_b(\Rd)$ we split $f=f^+-f^-$ and observe that the solution to problem \eqref{dreams-1} is the sum of the solutions $u_{n,+}$ and $u_{n,-}$ of this problem
corresponding to $f^+$ and $f^-$ respectively. Since the sequences $(u_{n,+})$ and $(u_{n,-})$ converge to the solutions to problem \eqref{dreams} with $f$ replaced respectively by $f^+$
and $f^-$, $u_n$ converges pointwise to a solution $u$ to problem \eqref{dreams} which belongs to $C([s,+\infty)\times\Rd)\cap C^{1,2}((s,+\infty)\times\Rd)$ and satisfies
 estimate \eqref{reality}. This completes the proof.
\end{proof}

\begin{remark}
\label{rem-2.1}
Some remarks are in order.
\begin{enumerate}[\rm (i)]
\item
If $f\ge 0$, then the solution to problem \eqref{dreams} is the minimal among all the solutions which belong to $C([s,+\infty)\times\Rd)\cap C^{1,2}((s,+\infty)\times\Rd)$ and
are bounded in each strip $[s,T]\times\Rd$. Indeed, if $w$ is any other solution, then, for any $n\in\N$, the function $z=w-u_n$  (where $u_n$ is as in the proof of the previous theorem)
solves the equation $D_tz=\A z$, $z(0,\cdot)\equiv 0$ and $z$ is nonnegative on $(s,+\infty)\times\partial B_n$. The maximum principle in \cite{krylov} and \cite[Theorem A.2]{ForMetPri04Gra} implies that
$z\ge 0$, i.e., $u_n\le w$ in $(s,+\infty)\times B_n$. Letting $n\to +\infty$ we conclude that $u\le w$ in $(s,+\infty)\times\Rd$.
\item
Hypothesis \ref{hyp-1}(iii) can not be avoided. Indeed, let us consider the one dimensional autonomous operator $\mathcal{A}=D_{xx}+c$ and assume that $c(x)$ diverges to $+\infty$ as $x\to +\infty$. Fix $n\in\N$, let
$M_n>0$ be such that $c(x)>n$ for any $x\in (M_n,+\infty)$ and suppose that $u\in C^{1,2}((0,+\infty)\times\R)\cap C([0,+\infty)\times\R)$ solves the equation $D_tu=\A u$ in $(0,+\infty)\times\R$, $u(0,\cdot)=1$ in $\R$ and $u(t,\cdot)$ is bounded in $\R$ for any $t>0$. Then, $D_tu\ge D_{xx}u+nu$ in $(0,+\infty)\times (M_n,+\infty)$. A comparison argument (see \cite{krylov} and \cite[Theorem A.2]{ForMetPri04Gra}) shows that $u(t,x)\ge e^{nt}v(t,x-M_n)$ for any $(t,x)\in (0,+\infty)\times (M_n,+\infty)$, where $v$ is the unique bounded classical solution to the Cauchy-Dirichlet problem
\begin{equation*}\left\{\begin{array}{lll}
D_tv(t,x)=D_{xx}v(t,x), \quad\, &t\in (0,+\infty), &x\in (0+\infty),\\[1mm]
v(t,0)=0,\quad\; &t\in (0,+\infty),\\[1mm]
v(0,x)=1,\quad\; &&x\in (0,+\infty).
\end{array}
\right.
\end{equation*}
It thus follows $\|u(t,\cdot)\|_{\infty}\ge e^{nt}v(t,1)$ for any $n\in\N$ and $t\in (0,+\infty)$. Since
$v(t,1)>0$ for any $t>0$, letting $n$ tend to $+\infty$ in the last inequality we get to a contradiction.
\item
Hypothesis \ref{hyp-1}(iv) is used to prove a variant of the classical maximum principle (see \cite{AngLor10Com,KunLorLun09Non}). Without such an assumption, the Cauchy problem \eqref{dreams}
may admit more than a unique solution $u\in C^{1,2}((s,+\infty)\times\Rd)\cap C([s,+\infty)\times\Rd)$ which is bounded in $[s,T]\times\Rd$ for any $T>s$. This was known since the middle of the last century in the one-dimensional case. Indeed, Feller provided
in \cite{feller} a complete characterization of the operators $\A=qD_{xx}+bD_x$ for which the elliptic equation $\lambda u-\A u=f\in C_b(\R)$ admits/does not admit for $\lambda>0$ a unique solution
$u\in C_b(\R)\cap C^2(\R)$. The characterization is given in term of integrability properties at infinity of the functions $Q$ and $R$ defined by
\begin{eqnarray*}
Q(x)=\frac{1}{q(x)W(x)}\int_0^xW(s)ds,\qquad\;\, R(x)=W(x)\int_0^x\frac{1}{q(s)W(s)}ds
\end{eqnarray*}
for any $x\in\R$, where $W$ is, up to a multiplicative constant, the wronskian determinant associated to the ordinary differential operator $qD_{xx}+bD_x$ i.e.,
\begin{eqnarray*}
W(x)=\exp\bigg (-\int_0^x\frac{b(s)}{q(s)}ds\bigg ),\qquad\;\,x\in\R.
\end{eqnarray*}
It turns out that the above elliptic equation admits a unique bounded solution $u\in C^2(\R)$ for any $f\in C_b(\R)$ if and only if
$R$ is not integrable either in a neighborhood of $-\infty$ and in a neighborhood of $-\infty$.
On the other hand, if $R$ is integrable both in a neighborhood of $+\infty$ and in a neighborhood of $-\infty$, then all the solutions
of the equation $\lambda u-qu''-bu'=f\in C_b(\R)$ are bounded. \\
Based on this remark, consider the operators $\A_+=D_{xx}+x^3D_x$ and $\A_-=D_{xx}-x^3D_x$.
In the first case,
\begin{eqnarray*}
Q_+(x)=e^{x^4/4}\int_0^xe^{-s^4/4}ds,\qquad\;\,R_+(x)=e^{-x^4/4}\int_0^xe^{s^4/4}ds
\end{eqnarray*}
for any $x\in\R$. The function $R_+$ belongs to $L^1((-\infty,0))\cap L^1((0,+\infty))$ and consequently, for any $f\in C_b(\Rd)$, the equation
$u-\A_+u=f\in C_b(\R)$ admits infinitely many bounded solutions $v\in C^2(\R)$. From any of such solution we obtain a solution $u$ of the parabolic equation $D_tu-\A u=0$ which belongs to $C^{1,2}([s,+\infty)\times\Rd)$
and is bounded in any strip $[s,T]\times\Rd$, simply by considering the function $u$ defined by $u(t,x)=e^{t-s}v(x)$ for any $(t,x)\in [s,+\infty)\times\Rd$.
Hypothesis \ref{hyp-1}(iv) is not satisfied by operator $\A_+$.\\
On the other hand, if we consider the operator $\A_-$ then the function $\varphi:\R\to\R$ defined by $\varphi(x)=1+x^2$ for any $x\in\R$ satisfies Hypothesis
\ref{hyp-1}(iv) and the Cauchy problem \eqref{dreams} is uniquely solvable for any $f\in C_b(\R)$.
\end{enumerate}
\end{remark}

In the rest of this paper we will always assume that Hypotheses \ref{hyp-1} hold true.
In view of Lemma \ref{lemma-2.1} and Remark \ref{rem-2.1}(iii), we can associate a family of bounded operators in $C_b(\Rd)$ to the operator $\A$: for any $f\in C_b(\Rd)$ and $I\ni s<t$,
$G(t,s)f$ is the value at $t$ of the unique classical solution to problem \eqref{dreams}. Estimate \eqref{reality} guarantees that each operator $G(t,s)$ is bounded
in $C_b(\Rd)$ and, again the variant of the classical maximum principle yields the evolution law
$G(t,s)=G(t,r)G(r,s)$ for any $I\ni s<r<t$.

As it has been proved in \cite{AngLor10Com}, a Green kernel can be associated with the evolution operator $G(t,s)$, i.e., there exists a function $g:\{(t,s)\in I\times I: t>s\}\times\Rd\times\Rd\to (0,+\infty)$
such that
\begin{equation}
(G(t,s)f)(x)=\int_{\Rd}f(y)g(t,s,x,y)dy,\qquad\;\,t>s\in I,\;\,x,y\in\Rd,\;\,f\in C_b(\Rd).
\label{kernel}
\end{equation}
For any fixed $s$, $t$ and $x$, the function $g(t,s,x,\cdot)$ belongs to $L^1(\Rd)$ and its $L^1$-norm is bounded from above by
$e^{c_0(t-s)}$. In particular, if $c\equiv 0$, then $g(t,s,x,y)dy$ is a probability measure.
From \eqref{kernel} it follows immediately that
\begin{equation}
[G(t,s)(fg)](x)\le [(G(t,s)|f|^p)(x)]^{\frac{1}{p}}[(G(t,s)|g|^q(x)]^{\frac{1}{q}},\qquad\;\,t>s\in I,\;\,x\in\Rd,
\label{fg}
\end{equation}
for any $f,g\in C_b(\Rd)$ and $p,q\in (1,+\infty)$ such that $1/p+1/q=1$. Moreover,
\begin{equation}
(G(t,s)f)(x)\le [(G(t,s)|f|^p)(x)]^{\frac{1}{p}},\qquad\;\,t>s\in I,\;\,x\in\Rd,\;\,f\in C_b(\Rd),
\label{fp}
\end{equation}
for any $p\in (1,+\infty)$, if $c_0\le 0$. For estimates for the Green function $g$, we refer the reader to
\cite{KLR,KLR-1}.

\section{Uniform estimates for the spatial derivatives of $G(t,s)f$ and consequences}
One powerful tool used to prove estimates for the derivatives of solutions to Cauchy problems (mainly in the whole space) is
the well celebrated Bernstein method (see \cite{bernstein}) which goes back to 1906, and the reiteration theorem (see \cite{triebel}).
The Bernstein method, used in the case of bounded coefficients, works well also in the case of unbounded coefficients, provided suitable both
algebraic and growth conditions on the coefficients of the operator $\A$ are prescribed. More precisely,
assume that
\medskip

\begin{hyps}
\begin{enumerate}
\item[\rm (i)]
the coefficients $q_{ij},b_j$ $(i,j=1,\ldots,d)$ and $c$ belong to
$C^{\alpha/2,k+\alpha}_{\rm loc}(I\times\Rd)$;
\item[\rm (ii)]
there exist two locally bounded functions $C_1, C_2:I\to\R$ such that
\begin{align*}
&|Q(t,x)x|+{\rm Tr}(Q(t,x))\le C_1(t)(1+|x|^2)\nu(t,x),\\
&\langle b(t,x),x\rangle \le C_2(t)(1+|x|^2)\nu(t,x),
\end{align*}
for any $t\in I$ and $x\in\Rd$;
\item[\rm (iii)]
there exist a locally bounded function $C:I\to\R$ and functions $r_0,r,\varrho:I\times\Rd\to\R$, with $\inf_{[a,b]\times\Rd}\varrho>0$ for any $[a,b]\subset I$ such that
$\langle ({\rm Jac}_xb)\xi,\xi\rangle\le r_0|\xi|^2$, $|D^{\beta}_xq_{ij}|\le C\nu$, $|D^{\delta}_x b_j|\le r$, $|D^{\eta}_xc|\le\varrho$
in $I\times\Rd$ for any $\{0,2\}\neq|\beta|\le k$, $1<|\delta|\le k$, $0\le |\eta|\le k$, $i,j=1,\ldots,d$ and $\xi\in\Rd$;
\item[\rm (iv)]
there exist locally bounded positive functions $L$ and $M$ such that
$r_0+L_kr+L\varrho^2\le M\nu$ in $I\times\Rd$,
where $L_1=0$, $L_2=d^{3/2}/\sqrt{8}$, $L_3=2/\sqrt{5}$ if $d=1$ and $L_3=\sqrt{d^3(d+1)/3}$ otherwise;
\item[\rm (v)]
if $k\ge 2$ then there exists a locally bounded function $K:I\to\R$ such that
\begin{align*}
\sijhk D_{hk}q_{ij}a_{ij}a_{hk}\le K\nu \shk a_{hk}^2
\end{align*}
in $I\times\Rd$, for any symmetric matrix $A=(a_{hk})$ and any $(t,x)\in
I\times\Rd$.
\end{enumerate}
\end{hyps}

Then, the Bernstein method allows to prove the following result.
\setcounter{theorem}{1}

\begin{theorem}[Theorem 2.4 of \cite{Lorenzi-1}]
\label{C0-C3}
Let Hypotheses $3.1(k)$ be satisfied. Then, for any $I\ni s<t$ and $f\in C_b(\Rd)$, the function
$G(t,s)f$ belongs to $C^k_b(\Rd)$. Moreover, for any $h,m\in\N$, with $h\le m\le k$ it holds that
\begin{equation}
\|G(t,s)f\|_{C^m_b(\R^d)}\le C_{h,m}(t-s)^{-\frac{m-h}{2}}\|f\|_{C^h_b(\R^d)},
\qquad\;\,t\in (s,T],
\label{stimasem}
\end{equation}
for any $f\in C_b^h(\R^d)$ and a positive constant $C_{h,m}$, independent of $s$, $T$ and $f$.
\end{theorem}

The reiteration theorem allows to extend the validity of estimate \eqref{stimasem} to the case when $h$ and $m$ are not integers. Finally, the evolution law
allows to extend \eqref{stimasem} to any $t>s$ up to adding an exponential factor
$e^{\omega_{h,m}(t-s)}$ in its right-hand side, for some nonnegative constant $\omega_{h,m}$, i.e., one can prove that
\begin{align*}
\|G(t,s)f\|_{C^m_b(\Rd)}\le C_{h,m}e^{\omega_{h,m}(t-s)}(t-s)^{-\frac{m-h}{2}}\|f\|_{C^h_b(\Rd)},\qquad\;\,s<t,
\end{align*}
for any $m\in (0,k)$, $f\in C^h_b(\Rd)$\footnote{Clearly, this method is too rough to provide us with the best constant $\omega_{h,m}$. Different arguments are used to improve
the asymptotic behaviour of the derivatives of the function $G(t,s)f$, as we will see in Section \ref{sect-5}.}. For further details, we refer the reader to \cite{Lorenzi-1}.

Using the above uniform estimates one can prove the following optimal Schauder estimates for the solution
to the Cauchy problem \eqref{nonom}.

\begin{theorem}[Theorem 2.7 of \cite{Lorenzi-1}]
Let Hypotheses $3.1(3)$ be
satisfied. Fix $\theta\in (0,1)$, $s\in I$ $g\in
C([s,T]\times\Rd)$, such that $\sup_{t\in [s,T]}\|g(t,\cdot)\|_{C^{\theta}_b(\Rd)}<+\infty$, and $f\in C^{2+\theta}_b(\Rd)$.
Then, problem \eqref{nonom} admits a unique solution $u\in C_b([s,T]\times\Rd)\cap C^{1,2}((s,T)\times\Rd)$.
Moreover, $u(t,\cdot)\in C^{2+\theta}_b(\Rd)$ for any
$t\in [s,T]$ and there exists a positive constant $C_0$ such that
\begin{align*}
\sup_{t\in [s,T]}\|u(t,\cdot)\|_{C_b^{2+\theta}(\Rd)}\le C_0\bigg (
\|f\|_{C^{2+\theta}_b(\Rd)}+\sup_{t\in [s,T]}\|g(t,\cdot)\|_{C^{\theta}_b(\Rd)}\bigg ).
\end{align*}
\end{theorem}

\begin{proposition}
\label{prop-3.4}
Under Hypotheses $3.1(k)$, if $f\in C^k_b(\Rd)$, then all the spatial derivatives of $G(\cdot,s)f$ up to the order $k$ are continuous in
$[s,+\infty)\times\Rd$.
\end{proposition}

\begin{proof}
Since the arguments used are independent of $k$, to fix the ideas we consider the case $k=3$. Clearly, we have just to prove the continuity on $\{s\}\times\Rd$ of
the spatial derivatives up to the third-order of the function $G(\cdot,s)$, since their continuity
in $(s,+\infty)\times\Rd$ is a classical result (see e.g., \cite{friedman}).

The proof is based on a localization argument as in the proof of Lemma \ref{lemma-2.1}. We fix $x_0\in \R^d$ and $m\in\N$ such that $x_0\in B_m$ and consider
a smooth cut-off function $\vartheta$ supported in $B_{m+1}$ and identically equal to $1$ in $B_m$.
Arguing as in the proof of Lemma \ref{lemma-2.1} we immediately see that
\begin{eqnarray*}
u(t,\cdot)=G_{m+1}(t,s)(\vartheta f)+\int_s^tG_{m+1}(t,r)\psi(r,\cdot)dr=u_1(t,\cdot)+u_2(t,\cdot)
\end{eqnarray*}
in $B_m$, for any $t\in (s,+\infty)$,
where $\psi=-G(\cdot,s)f(\A-c)\vartheta-2\langle Q\nabla_xG(\cdot,s)f,\nabla\vartheta\rangle$.

By classical results, the function $u_1$ and its spatial derivatives up to the third order
are continuous in $[s,+\infty)\times B_{m+1}$. As far as $u_2$ is concerned, we observe that
there exists a positive constant $C$, depending also on $s$, such that
\begin{eqnarray*}
\|G_{m+1}(t,s)\psi\|_{C^3(\overline{B_{m+1}})}\le \frac{C}{\sqrt{t-s}}\|\psi\|_{C^2(\overline{B_{m+1}})},\qquad\;\,t\in (s,s+1),
\end{eqnarray*}
for any $\psi\in C^2(\overline{B_{m+1}})$. Since $\|\psi(r,\cdot)\|_{C^2(\overline{B_{m+1}})}\le \widetilde C\|G(t,\cdot)f\|_{C^3(\overline{B_{m+1}})}$, using
Theorem \ref{C0-C3}, we immediately deduce that
$\|\psi(r,\cdot)\|_{C^2(\overline{B_{m+1}})}\le \overline C\|f\|_{C^3_b(\Rd)}$ for any $r\in (s,s+1)$, where $\widetilde C$ and $\overline C$ are positive
constants independent of $r$. Thus, we conclude that
\begin{align*}
\|u_2(t,\cdot)\|_{C^3(\overline{B_{m+1}})}\le&\int_s^t\|G_{m+1}(t,r)\psi(r,\cdot)\|_{C^3(\overline{B_{m+1}})}dr\\
\le& C\overline C\|f\|_{C^k_b(\Rd)}\int_s^t(t-r)^{-\frac{1}{2}}dr\\
= &2C\overline C\|f\|_{C^3_b(\Rd)}\sqrt{t-s}
\end{align*}
for any $t\in (s,s+1)$. Hence, letting $t\to s^+$ we conclude that $u(t,\cdot)$ and its spatial derivatives up to the third order vanish uniformly in $B_m$
as $t\to s^+$. By the arbitrariness of $m$ the claim follows.
\end{proof}

\section{Pointwise estimates for the derivatives of $G(t,s)f$}
\label{sect-4}

The pointwise gradient estimates for $G(t,s)f$ plays  an important role in the
study of many properties of the evolution operator, as we have already stressed in the Introduction. All these properties will be investigated
in Section \ref{sect-5}.
Here, we prove some pointwise estimates for the
derivatives (up to the third order) of $G(t,s)f$.

Throughout this section, we assume the following set of assumptions.
\begin{hyps0}
Hypotheses $3.1(k)$ are satisfied with the following differences:
\begin{itemize}
\item
$|D^{\beta}_xq_{ij}|\le C\nu^{\gamma}$
in $I\times\Rd$ for any $i,j=1,\ldots,d$, some positive constant $C$ and some $\gamma\in (0,1)$;
\item
$r_0+L_kr+L\rho^2\le M\nu^{\gamma}$ in $I\times\Rd$ for any $k=1,2,3$ and some constants $L>0$ and $M\in\R$, where
the constant $L_k$ is defined in Hypothesis $3.1(k)(iv)$;
\item
Hypothesis $3.1(k)(v)$ is satisfied with $K\nu$ being replaced by $K\nu^{\gamma}$, $K$ being a real constant.
\end{itemize}

\end{hyps0}

The scheme of this section is the following: first we prove estimate \eqref{aaaa} (with $h=k$) for any $k=1,2,3$ and $p \in(1,+\infty)$.
Next, strengthening the assumptions on the coefficients of the operator $\A$ we prove \eqref{aa}.
Note that if this estimate holds true, then, taking as $f=\mathds{1}$,
we conclude that $\nabla_xG(t,s)\mathds{1}$ identically vanishes in $\Rd$, that is $G(t,s)\mathds{1}=\psi(t)$ for any $t>s$ and some function
$\psi\in C([s,+\infty))\cap C^1((s,+\infty))$, which solves the equation
$\psi'(t)=c(t,x)\psi(t)$ and satisfies the condition $\psi(s)=1$. Since $G(\cdot,s)\mathds{1}$ is positive in $(s,+\infty)\times\Rd$, it follows that
$\psi(t)$ is positive for any $t>s$. We thus conclude that $c(t,x)=\psi'(t)/\psi(t)$ for any $(t,x)\in (s,+\infty)\times\Rd$, i.e., $c$ is independent of $x$.
Hence, if $u$ solves the Cauchy problem \eqref{dreams}, then the function
$w:[s,+\infty)\times\Rd\to\R$, defined by
\begin{eqnarray*}
w(t,x)=\exp\bigg (-\int_s^tc(r)dr\bigg )u(t,x),\qquad\;\,(t,x)\in [s,+\infty)\times\Rd,
\end{eqnarray*}
has the same degree of smoothness of the function $u$ and solves the Cauchy problem \eqref{dreams} with
$\A$ being replaced by the operator $\A_0={\rm Tr}(QD^2)+\langle b,\nabla\rangle$.
For this reason in the proof of Theorem \ref{teo-sharzan} we confine ourselves to the case when
$c \equiv 0$.

Next, we deal with the case $p=1$ in \eqref{aa}. As it has been explained in the Introduction, to get such an estimate we require that the diffusion coefficients do not depend on the space variable.
Finally, we prove estimate \eqref{aaaa} with $h=k-1$ and $k=1,2,3$ showing that
$\Gamma_{p,k-1,k}^{(2)}(r)\sim c_{p,k}r^{-p/2}$ as $r\to 0^+$, for some positive constant $c_{p,k}$.
As a byproduct, estimate \eqref{aaaa} follows in its full generality. In particular,
$\Gamma_{p,h,k}^{(2)}(r)\sim c_{p,k}'r^{-(k-h)p/2}$ as $r\to 0^+$, for some positive constant
$c_{p,k}$.
All these estimates have been proved in \cite{BerLor05} in the autonomous case when $c \equiv 0$.

To prove the above estimates in the general case we need a preliminary result.

\begin{lemma}
\label{lem-bonaventura}
Let the sequence $(c_n)\subset C^{\alpha/2,\alpha}_{\rm loc}(I\times\Rd)\cap C_b(I\times\Rd)$  converges to $c$ locally uniformly in $(s,+\infty)\times\Rd$ as $n$ tends to $+\infty$ and $c_n(t,x)\le M$ for any $n\in\N$, $(t,x)\in (s,+\infty)\times\Rd$ and some constant $M$.
For any $n\in\N$, $s\in I$ and $f\in C_b(\Rd)$, let $u_n$ solve the Cauchy problem
\begin{equation}
\left\{
\begin{array}{ll}
D_tu={\rm Tr}(QD^2_xu)+\langle b,\nabla_xu\rangle+c_nu, & {\rm in}~(s,+\infty)\times\Rd,\\[1mm]
u(s,\cdot)=f, & {\rm in}~\Rd.
\end{array}
\right.
\label{dreams-n}
\end{equation}
Further, denote by $u\in C_b([s,+\infty)\times\Rd)\cap C^{1+\alpha/2,2+\alpha}((s,+\infty)\times\Rd)$ the
solution to the Cauchy problem \eqref{dreams}, provided by Lemma $\ref{lemma-2.1}$.
Then $u_n$ converges to $u$ in $C^{1,2}([a,b]\times K)$ for any $[a,b]\subset (s,+\infty)$ and any compact set $K\subset\Rd$.

Finally, if Hypothesis $3.1(1)(i)$ is satisfied, $c_n\in C^{\alpha/2,1+\alpha}_{\rm loc}(I\times\Rd)$
and $\nabla_xc_n$ converges to $\nabla c$
locally uniformly in $I\times\Rd$, then $D^3_{ijh}u_n$ converges to $D^3_{ijh}u$ locally uniformly in $(s,+\infty)\times\Rd$ for any $i,j,h=1,\ldots,d$.
\end{lemma}

\begin{proof}
The proof of the first part follows the same lines as in the proof of Lemma \ref{lemma-2.1}; hence, we just sketch it.
By Lemma \ref{lemma-2.1}, the Cauchy problem \eqref{dreams-n} admits, for any $n\in\N$ a solution $u_n$ which satisfies the estimate
\begin{equation}
\|u_n(t,\cdot)\|_{\infty}\le e^{M(t-s)}\|f\|_{\infty},\qquad\;\,t>s,
\label{ariosto}
\end{equation}
where $M$ is as in the statement. The interior Schauder estimate in \cite[Theorem 4.10.1]{LadSolUra68Lin} and a diagonal argument imply that there exists a subsequence $(u_{n_k})$ which, as $k\to +\infty$, converges in $C^{1,2}([a,b]\times K)$ to a function $u\in C^{1+\alpha/2,2+\alpha}_{\rm loc}((s,+\infty)\times\Rd)$ for
any $[a,b]\subset (s,+\infty)$ and any compact set $K\subset\Rd$, and $u$ solves the equation
$D_tu=\A u$ in $(s,+\infty)\times\Rd$.

The same arguments used in the proof of Lemma \ref{lemma-2.1} and applied to the function $u_{n_k}$ show
that
\begin{eqnarray*}
u_{n_k}(t,\cdot)=G_{m+1}(t,s)f-\int_{s}^tG_{m+1}(r,s)\overline{\psi}_{n_k}(r,\cdot)dr,\qquad\;\,t>s,
\end{eqnarray*}
in $B_m$, where $G_{m+1}(t,s)$ denotes the evolution operator associated with the realization in $C_b(B_{m+1})$ of the operator $\A$ with homogenous Dirichlet boundary conditions,
$\overline{\psi}_{n_k}=\psi_{n_k}+\vartheta(c-c_{n_k})u_{n_k}$, $\psi_{n_k}$ being as in the proof of Lemma \ref{lemma-2.1}. Since $|G_{m+1}(r,s)\overline{\psi}_{n_k}(r,\cdot)|\leq
C(r-s)^{-1/2}\|f\|_\infty$ in $B_{m+1}$ for any $r\in (s,s+1)$ and some positive constant $C$, independent of $k$, as in the proof of Lemma \ref{lemma-2.1} we conclude that $u$ can be extended by continuity in $[s,+\infty)\times\Rd$ by setting
$u(s,\cdot)=f$.

To complete the proof, we assume that the coefficients of the operator $\A$ are once continuously differentiable with respect to the spatial variables
in $I\times\Rd$ with derivatives which belong to $C^{\alpha/2,\alpha}_{\rm loc}(I\times\Rd)$. Then, by the proof of \cite[Theorem 3.10]{friedman}, it follows that
there exists a positive constant $C$, independent of $k$ such that
$\|D^3_xu_{n_k}\|_{C^{\alpha/2,\alpha}([a,b]\times K)}\le C$ for any $[a,b]\subset (s,+\infty)$ and any compact set $K\subset\Rd$.
Hence, up to a subsequence, all the third-order derivatives of $u_{n_k}$ converge uniformly in $[a,b]\times K$, and
clearly they converge to the corresponding third-order spatial derivative of $u$.
Since $[a,b]$ and $K$ have been arbitrarily fixed, the proof is complete.
\end{proof}

\begin{theorem}\label{cucu}
Let Hypotheses $4.1(k)$ be satisfied. Then, estimate \eqref{aaaa} holds true, with $h=k$ and
$\Gamma_{p,k,k}(r)=e^{\sigma_{k,p}r}$ for any $r>0$
where
\begin{equation}
\sigma_{k,p}r=\bigg [p\sup_{I\times \Rd}[(1-p)\nu+c_k(p)\nu^{\gamma}]+c_0(p-1)+pc_{d,k}\bigg ]^+
\label{sigma-k-p}
\end{equation}
if $p\in (1, 2]$, $c_k(p)$ and $c_{d,k}$ being positive constants explicitly determined $($see the proof\hskip 3pt$)$ and
$\sigma_{k,p}=p\sigma_{k,2}/2$ if $p>2$.
\end{theorem}
\begin{proof}
We split the proof into three steps. In the first one we prove the estimate when $p \in(1,2]$ and $c$ is bounded. In the second step, using Lemma
\ref{lem-bonaventura} we remove the assumption on the boundedness of $c$. Finally in the last one we obtain the claim also in the case $p>2$.
To simplify the notation, throughout the proof, we set
\begin{align*}
\begin{array}{ll}
{\mathscr Q}_0(\zeta)=\langle Q\nabla_x\zeta,\nabla_x\zeta\rangle,\q &{\mathscr Q}_1(\zeta)=\ds\sijh D_hq_{ij}D_h\zeta D_{ij}\zeta,\\[2mm]
{\mathscr Q}_2(\zeta)=\ds\sijhk D_{hk}q_{ij}D_{ij}\zeta D_{hk}\zeta,\q &{\mathscr B}_1(\zeta)=\langle ({\rm Jac}_xb)\nabla_x\zeta,\nabla_x\zeta\rangle\\[2mm]
{\mathscr B}_2(\zeta)=\ds\sijh D_{jh}b_iD_i\zeta D_{jh}\zeta, & {\mathscr C}_1(\zeta)=\zeta\langle\nabla_xc,\nabla\zeta\rangle\\[5mm]
{\mathscr C}_2(\zeta)=\zeta{\rm Tr}(D^2_xcD^2_x\zeta)
\end{array}
\end{align*}
for any smooth enough function $\zeta:\Rd\to\R$.

{\em Step 1.} Let $p \in (1,2]$ and assume that $c$ is bounded.
We first consider the case $j=3$. For simplicity, we set $u=G(\cdot, s)f$ and, for any $\tau>0$, we introduce the function $w_\tau=(\sum_{k=0}^3|D^k_x u|^2+\tau)^{p/2}$, which is positive and belongs to
$C_b([s, +\infty)\times \Rd)\cap C^{1,2}_{\rm{loc}}((s, +\infty)\times \Rd)$, by virtue of Proposition \ref{prop-3.4} and Theorem \ref{C0-C3}.
Moreover, it solves the differential equation $D_t w_\tau-\A w_\tau= \psi_{\tau}$ in $(s,+\infty)\times \Rd$ where
\begin{align*}
\psi_\tau=&pw_\tau^{1-\frac{2}{p}}\sum_{i=1}^4{\mathscr J}_i(u)+p(2-p)w_{\tau}^{1-\frac{4}{p}}\langle Q\xi_u,\xi_u\rangle+(p-1)cw_\tau-pc\tau w_{\tau}^{1-\frac{2}{p}},
\end{align*}
and
\begin{align*}
{\mathscr J}_1(u)=&-{\mathscr Q}_0(u)-\si{\mathscr Q}_0(D_iu)-\sij {\mathscr Q}_0(D_{ij}u)-\sijh {\mathscr Q}_0(D_{ijh}u)\\
&+{\mathscr B}_1(u)+2\si {\mathscr B}_1(D_iu)+3\sij {\mathscr B}_1(D_{ij}u),\\[2mm]
{\mathscr J}_2(u)=&{\mathscr Q}_1(u)+2\si {\mathscr Q}_1(D_iu)+3\sij {\mathscr Q}_1(D_{ij}u)+{\mathscr Q}_2(u)\\
&+3\si {\mathscr Q}_2(D_iu)+\sijhkl D_{hkl}q_{ij}D_{ij}uD_{hkl}u,\\[2mm]
{\mathscr J}_3(u)=&{\mathscr B}_2(u)+3\si {\mathscr B}_2(D_iu)+\sijhk D_{jhk}b_iD_iuD_{jhk}u,\\[2mm]
{\mathscr J}_4(u)=&{\mathscr C}_1(u)+2\si {\mathscr C}_1(D_iu)+3\sij {\mathscr C}_1(D_{ij}u)+{\mathscr C}_2(u)\\
&+3\si {\mathscr C}_2(D_iu)\!+\!u\sijh D_{ijh}cD_{ijh}u,
\end{align*}
\begin{eqnarray*}
\xi_u=u\nabla_x u+\si D_iu\nabla_xD_iu+\sij D_{ij}u\nabla_xD_{ij}u+\sijh D_{ijh}u\nabla_xD_{ijh}u.
\end{eqnarray*}
Here and below, all the equalities and inequalities that we write are meant in $(s,+\infty)\times\Rd$.

Let $\alpha$ and $\beta$, with $|\alpha|, |\beta|\le 3$, be fixed.
The Cauchy-Schwarz inequality applied twice, yields
\begin{align}
&\sij q_{ij} \sum_{|\alpha|=h}D^\alpha_x u D_iD^\alpha_x u
\sum_{|\beta|=k}D^\beta_x u D_jD^\beta_x u\notag\\
\le &\sum_{|\alpha|=h}|D^\alpha_x u|({\mathscr Q}_0(D^{\alpha}_xu))^{\frac{1}{2}}
 \sum_{|\beta|=k} |D^\beta_x u|({\mathscr Q}_0(D^{\beta}_xu))^{\frac{1}{2}}
\notag\\
\le &\ds |D^h_xu||D^k_xu|\bigg (\sum_{|\alpha|=h}{\mathscr Q}_0(D^{\alpha}_xu)\bigg )^{\frac{1}{2}}
\bigg (\sum_{|\beta|=k}{\mathscr Q}_0(D^{\beta}_xu)\bigg )^{\frac{1}{2}}.
\label{cauchy-derivate}
\end{align}
In view of \eqref{cauchy-derivate} we get
\begin{align*}
\langle Q\xi_u,\xi_u\rangle
\le &\bigg [\ds |u|({\mathscr Q}_0(u))^{\frac{1}{2}}+|D_xu|\bigg (\si{\mathscr Q}_0(D_iu)\bigg )^{\frac{1}{2}}
+|D^2_xu|\bigg (\sij {\mathscr Q}_0(D_{ij}u)\bigg )^{\frac{1}{2}}\notag\\
&\;\,+|D^3_xu|^2\bigg (\sijh {\mathscr Q}_0(D_{ijh}u)\bigg )^{\frac{1}{2}}\bigg ]^2\notag\\
\le & w^{\frac{2}{p}}_{\tau}\bigg ({\mathscr Q}_0(u)+\si{\mathscr Q}_0(D_iu)+\sij {\mathscr Q}_0(D_{ij}u)
+\sijh {\mathscr Q}_0(D_{ijh}u)\bigg ).
\end{align*}
Hence, taking Hypotheses \ref{hyp-1}(ii) and 3.1(3)(iii) into account, we can estimate the ``good'' terms in the definition of $\psi_\tau$ as follows:
\begin{align}
{\mathscr J}_1(u)\le \sum_{k=1}^3[(1-p)\nu+kr_0]|D^k_xu|^2+(1-p)\nu|D^4_xu|^2.
\label{J-1}
\end{align}

The other terms in the definition of the function $\psi_{\tau}$ are estimated using Hypotheses 3.1(3)(iii), 3.1(3)(v) (where, now, $C$ and $K$ are constants)
and the Cauchy-Schwarz inequality. We get
\begin{align}
&{\mathscr Q}_1(\zeta)\le C\nu^{\gamma} \sij |D_h\zeta||D_{ij}\zeta|\le Cd^{\frac{3}{2}}\nu|\nabla_x\zeta||D^2_x\zeta|\notag\\
&\phantom{|{\mathscr Q}_1(\zeta)|}\le\frac{Cd^2}{4\varepsilon}\nu^{\gamma}|\nabla_x\zeta|^2+Cd\varepsilon\nu^{\gamma}|D^2_x\zeta|^2,
\label{estim-Q1}
\\[2mm]
&{\mathscr Q}_2(\zeta)\le K\nu^{\gamma}|D^2\zeta|^2,
\label{estim-Q2}
\\[1mm]
&\sijhkl D_{hkl}q_{ij}D_{ij}\zeta D_{hkl}\zeta\le \frac{Cd^3}{4\varepsilon}\nu^{\gamma}|D^2\zeta|^2+Cd^2\varepsilon\nu^{\gamma}|D^3\zeta|^2\notag
\end{align}
for any smooth enough function $\zeta:\Rd\to\R$ and $\varepsilon>0$, which shows that
\begin{align*}
{\mathscr J}_2(u)\le & \frac{Cd^2}{4\ve}\nu^\gamma |D_xu|^2+\bigg (C\ve d+\frac{Cd^2}{2\ve}+\frac{Cd^3}{4\ve}+K\bigg )\nu^\gamma |D_x^2u|^2\notag\\
&+\bigg ( 2C\ve d+\frac{3Cd^2}{4\ve}+
\ve Cd^2+3K\bigg )\nu^\gamma |D^3_xu|^2\notag+3C\ve d\nu^\gamma|D^4_xu|^2.
\end{align*}
Similarly,
\begin{align}
&{\mathscr B}_2(\zeta)\le\frac{d^2}{4\varepsilon_1}r|\nabla_x\zeta|^2+d\varepsilon_1r|D^2_x\zeta|^2,
\label{estim-B2}\\[1mm]
&\bigg |\sijhk D_{jhk}b_iD_iuD_{jhk}\zeta\bigg |\le \frac{d^3}{4\varepsilon_1}r|\nabla\zeta|^2+d\varepsilon_1 r|D^3\zeta|^2\notag
\end{align}
for any $\zeta$ as above and any $\varepsilon_1>0$. Hence,
\begin{equation*}
{\mathscr J}_3(u)\le
\frac{d^2}{4\ve_1}(d+1)r|\nabla_xu|^2+dr\bigg (\ve_1 +\frac{3d}{4\ve_1}\bigg )|D^2_xu|^2+4d\ve_1r|D^3_xu|^2
\end{equation*}
for any $\varepsilon_1>0$.
Further,
\begin{align*}
&{\mathscr C}_1(\zeta)\le \frac{1}{4\varepsilon_2}\zeta^2+d\varepsilon_2\rho^2|\nabla_x\zeta|^2,\qquad\;\,{\mathscr C}_2(\zeta)\le \frac{d}{4\varepsilon_2}\zeta^2+d\varepsilon_2\rho^2|D^2\zeta|^2\\
& \zeta\sijh D_{ijh}cD_{ijh}\zeta\le \frac{d^3}{4\varepsilon_2}\zeta^2+\varepsilon_2\rho^2|D^3\zeta|^2
\end{align*}
for any smooth enough function $\zeta:\Rd\to\R$ and $\varepsilon_2>0$.
It thus follows that
\begin{align*}
{\mathscr J}_4(u)\le &\frac{d^3+d+1}{4\varepsilon_2}u^2+\bigg (d\varepsilon_2\rho^2+
\frac{1}{2\varepsilon_2}+\frac{3d}{4\varepsilon_2}\bigg )|\nabla_xu|^2
+\bigg (3d\varepsilon_2\rho^2+\frac{3}{4\varepsilon_2}\bigg )|D^2_xu|^2\notag\\
&+(6d+1)\varepsilon_2\rho^2|D^3_xu|^2\notag\\
\le &\frac{3d^3}{4\varepsilon_2}u^2+\bigg (d\varepsilon_2\rho^2+
\frac{5d}{4\varepsilon_2}\bigg )|\nabla_xu|^2
+\bigg (3d\varepsilon_2\rho^2+\frac{3}{4\varepsilon_2}\bigg ) |D^2_xu|^2\notag\\
&+7d\varepsilon_2\rho^2|D^3_xu|^2.
\end{align*}

Finally, $-c\tau w^{1-2/p}\le \|c\|_{\infty}\tau^{p/2}$.

Summing up, from all the previous estimates it follows that we can make
nonnegative the coefficient in front of $|D^4_xu|^2$ by taking
$\ve=(p-1)\nu_0^{1-\gamma}/(3Cd)$. With this choice of $\ve$, we get
\begin{align*}
\psi_\tau\le &p\bigg\{\bigg [(1-p)\nu+\frac{3C^2d^3}{4(p-1)}\nu_0^{\gamma-1}\nu^\gamma
+r_0+r\frac{d^3+d^2}{4\ve_1}+\bigg (d\varepsilon_2\rho^2+
\frac{5d}{4\varepsilon_2}\bigg )\bigg ]|\nabla_xu|^2\notag\\
&\phantom{\Bigg\{}+\bigg [(1-p)\nu\!+\! \Big (\frac{p-1}{3}\nu_0^{1-\gamma}\!+\!\frac{3C^2d^3}{2(p-1)}\nu_0^{\gamma-1}\!+\!\frac{3C^2d^4}{4(p-1)}\nu_0^{\gamma-1}\!+\!K\Big )\nu^\gamma\!+\!2r_0\notag\\
&\phantom{\Bigg\{}\qq+rd\bigg (\ve_1 +\frac{3d}{4\ve_1}\bigg )+\bigg (3d\varepsilon_2\rho^2+\frac{3}{4\varepsilon_2}\bigg )\bigg ]|D^2_xu|^2\notag\\
&\phantom{\Bigg\{}+\bigg [(1\!-\!p)\nu \!+\! \bigg (\frac{2(p-1)}{3}\nu_0^{\gamma-1}\!+\!\frac{3C^2d^3}{4(p-1)}\nu_0^{\gamma-1}\!+\!\frac{(p-1)d}{3}\nu_0^{\gamma-1}+3K\bigg )\nu^\gamma\notag\\
&\phantom{\Bigg\{}\qq+3r_0+4r\ve_1d+7d\varepsilon_2\rho^2\bigg ]|D^3_xu|^2
\bigg\}w_\tau^{1-\frac{2}{p}}+(p-1)c_0w_{\tau}\\
&+\frac{3pd^3}{4\ve_2}u^2w_\tau^{1-\frac{2}{p}}+p\|c\|_{\infty}\tau^{\frac{p}{2}}.
\end{align*}
Next, we choose $\ve_1=3\sqrt{5}/10$, if $d=1$, and $\ve_1=\sqrt{3d(d+1)}/4$ otherwise (which is the point where
the function $x\mapsto d \max\{(d^2+d)/(4x),(4x^2+3d)/(8x),4x/3\}$ attains its minimum value) and $\ve_2=3L/(7d)$,
to get
\begin{align*}
\psi_\tau\le &p(p-1)w_\tau^{1-\frac{2}{p}}(c_3(p)\nu^\gamma-\nu)\sum_{j=1}^3|D_x^ju|^2+[(p-1)c_0+ pc_{d,3}]w_\tau
+p\|c\|_{\infty}\tau^{\frac{p}{2}},
\end{align*}
where $c_{d,3}=7d^2(5\vee3d^2)/(12L)$, $c_3(p)=\max\{{\mathscr K}_{i,p}, i=1,2,3\}$ and
\begin{align*}
{\mathscr K}_{1,p}=& \frac{3C^2d^3}{4(p-1)}\nu_0^{\gamma-1}+M,
\\[1mm]
{\mathscr K}_{2,p}=& \frac{1}{3}(p-1)\nu_0^{1-\gamma}
+ \frac{3C^2d^3 (d+2)}{4(p-1)}\nu_0^{\gamma-1}+K+2M, \\[2mm]
{\mathscr K}_{3,p}=& \frac{1}{3}(d+2)(p-1)\nu_0^{1-\gamma}
+ \frac{3C^2d^3 }{4(p-1)}\nu_0^{\gamma-1} +3K+3M,
\end{align*}
$M$ being the constant in Hypothesis 3.1(3)(iv).

Hence, the function $w_{\tau}$ satisfies the differential inequality
$D_tw_{\tau}-\A w_{\tau}\le \s_{3,p}w_\tau+p\|c\|_{\infty}\tau^{\frac{p}{2}}$
in $(s,+\infty)\times\Rd$, where $\s_{3,p}$ is as in the statement.

Now, we set $z_\tau(t,x)=e^{-\s_{3,p}(t-s)}(w_\tau(t,x)-p\|c\|_{\infty}\tau^{p/2}(t-s))$, for any $(t,s)\in (s,+\infty)\times\Rd$,
and observe that the function $z_\tau$ solves the problem
\begin{equation*}
\left\{
\begin{array}{lll}
D_tz_\tau(t,x)\le \mathcal{A}z_\tau(t,x),  &t>s,\;\;\,x\in\Rd,\\[2mm]
z_\tau(s,x)=
\bigg (\ds\sum_{k=0}^3|D^kf(x)|^2+\tau\bigg )^{\frac{p}{2}}, &\qq\q\;\, x\in\Rd.
\end{array}
\right.
\label{pb:point1}
\end{equation*}
Then, the maximum principle in \cite[Proposition 2.2]{AngLor10Com} implies that,
\begin{eqnarray*}
z_{\tau}\le G(\cdot,s)(|f|^2+|\nabla f|^2+|D^2f|^2+|D^3f|^2+\tau)^{\frac{p}{2}}
\end{eqnarray*}
in $(s,+\infty)\times\Rd$, whence estimate \eqref{aaaa}, with $k=h=3$, follows letting $\tau\to 0^+$.

To get \eqref{aaaa} when $h=k=2$, it suffices to apply the previous
arguments to the function
$w_\tau=(u^2+|\nabla_xu|^2+|D^2_xu|^2+\tau)^{p/2}$.
Arguing as above and taking $\varepsilon=(p-1)\nu_0^{1-\gamma}/(2Cd)$,
we prove that $D_tw_{\tau}-\A w_{\tau}\le\Psi_{\tau}$, where
\begin{align*}
\Psi_\tau= &p\Bigg\{\bigg [(1-p)\nu+\frac{C^2d^3}{2(p-1)}\nu_0^{\gamma-1}\nu^\gamma+r_0+r\frac{d^2}{4\ve_1}+\bigg (d\varepsilon_2\rho^2+
\frac{1}{2\varepsilon_2}\bigg )\bigg ]|\nabla_xu|^2\notag\\
&\phantom{\Bigg\{}+\bigg [(1-p)\nu+ \Big (\frac{p-1}{2}\nu_0^{1-\gamma}+\frac{C^2d^3}{p-1}\nu_0^{\gamma-1}+K\Big )\nu^\gamma\\
&\phantom{\Bigg\{+\bigg [}+2r_0+\ve_1rd+3d\varepsilon_2\rho^2\bigg ]|D^2_xu|^2\bigg\}w^{1-\frac{2}{p}}\\
&+\bigg [(p-1)c_0+\frac{pd}{2\ve_2}\bigg ]w_{\tau}+p\|c\|_{\infty}\tau^{\frac{p}{2}},
\end{align*}
Then, we take $\varepsilon_1=\sqrt{d/2}$, to minimize the maximum between $d^2(4\ve_1)^{-1}$ and $\ve_1d/2$,
and $\ve_2=2L/(3d)$. We thus get
\begin{align*}
\Psi_\tau\le &p\Bigg\{\bigg [(1-p)\nu+\frac{C^2d^3}{2(p-1)}\nu_0^{\gamma-1}\nu^{\gamma}+M\nu^{\gamma}+
\frac{3d}{4L}\bigg ]|\nabla_xu|^2\notag\\
&\phantom{\Bigg\{}+\bigg [(1-p)\nu+ \Big (\frac{p-1}{2}\nu_0^{1-\gamma}+\frac{C^2d^3}{p-1}\nu_0^{\gamma-1}+K+2M\Big )\nu^\gamma\bigg ]|D^2_xu|^2\bigg\}w_{\tau}^{1-\frac{2}{p}}\\
&+(p-1)c_0w_{\tau}+\frac{3pd^2}{4L}u^2w_{\tau}^{1-\frac{2}{p}}+p\|c\|_{\infty}\tau^{\frac{p}{2}}.
\end{align*}
Hence, \eqref{aaaa}, with $h=k=2$, follows with
$c_{d,2}=3d^2(4L)^{-1}$ and
\begin{equation*}
c_2(p)=\max\bigg\{\frac{C^2d^3\nu_0^{\gamma-1}}{2(p-1)}+M,
\frac{p-1}{2}\nu_0^{1-\gamma}+\frac{C^2d^3\nu_0^{\gamma-1}}{p-1}+K+2M\bigg\}.
\end{equation*}

Finally, to get \eqref{aaaa} with $h=k=1$, we consider the function $w_{\tau}=(u^2+|\nabla_xu|^2+\tau)^{p/2}$ which satisfies
the inequality $D_tw_{\tau}\le\A w_{\tau}+\Psi_{\tau}$, where
\begin{align*}
\Psi_{\tau}=&p\bigg\{\bigg [(1-p)\nu+r_0+r+\frac{Cd^2}{4\ve}\nu^{\gamma}+d\ve_2\rho^2\bigg ]|\nabla_xu|^2\\
&\phantom{\bigg\{}
+[(1-p)\nu+Cd\varepsilon\nu^{\gamma}]|D^2_xu|^2\bigg\}w_{\tau}^{1-\frac{2}{p}}\\
&+\frac{p}{4\ve_2}u^2w_{\tau}^{1-\frac{2}{p}}+(p-1)c_0w_{\tau}+p\|c\|_{\infty}\tau^{\frac{p}{2}}
\end{align*}
for any $\ve,\ve_2>0$. We take $\varepsilon=(p-1)\nu_0^{1-\gamma}(Cd)^{-1}$ and $\ve_2=L/d$ to get
\begin{align*}
\psi_{\tau}\le &p\bigg\{\bigg [(1-p)\nu+\bigg (\frac{C^2d^3}{4(p-1)}\nu_0^{\gamma-1}+M\bigg )\nu^{\gamma}\bigg ]|\nabla_xu|^2
\bigg\}w_{\tau}^{1-\frac{2}{p}}\\
&+\frac{pd}{4L}u^2w_{\tau}^{1-\frac{2}{p}}+(p-1)c_0w_{\tau}+p\|c\|_{\infty}\tau^{\frac{p}{2}}.
\end{align*}
Thus, \eqref{aaaa} (with $h=k=1$) follows with $c_{d,1}=d/(4L)$
 and $c_1(p)=\frac{C^2d^3\nu_0^{\gamma-1}}{4(p-1)}+M$.

{\em Step 2.}
Here we prove estimate \eqref{aaaa} for $p \in (1,2]$ in the general case. Just to fix ideas, we consider the case $k=3$.
We introduce two sequences $(\vartheta_n)$ and $(\psi_n)$ of smooth cut-off functions such that $\chi_{B_n}\le\vartheta_n\le\chi_{B_{n+1}}$
and $\chi_{(s+2/n,s+4n)}\le\psi_n\le \chi_{(s+1/n,s+8n)}$ for any $n\in\N$.
Without loss of generality, we can assume that $\|D^{\beta}\vartheta_n\|_{\infty}\le C_0n^{-|\beta|}$ for any $|\beta|\le 3$ and some positive
constant $C_0$.
For any $n\in\N$ we set $c_n(t,x)=\psi_n(t)\vartheta_n(x)c(t,x)$ for any $(t,x)\in (s,+\infty)\times\Rd$. Clearly each function $c_n$ is bounded. Moreover,
$|D^{\eta}_xc_n|\le \varrho_n:=(1+C_1n^{-1})\varrho$ for any $n\in\N$, $|\eta|=1,2,3$ and some positive constant $C_1$.
Note that in view of Hypothesis 3.1(3)(iv) (where, now, $L$ and $M$ are constants) it follows that
$r_0+L_3r+L^{(n)}_3\varrho_n^2\le M\nu^{\gamma}$ in $(s,+\infty)\times\Rd$, where $L_3^{(n)}=L[1+(C_1^2+2C_1)n^{-1}]^{-1}$.
By Step 1 it follows that the solution $u_n=G^{(n)}(\cdot,s)f$ to problem
\begin{eqnarray*}
\left\{
\begin{array}{ll}
D_tu_n={\rm Tr}(QD^2_xu_n)+\langle b,\nabla_xu_n\rangle+c_nu_n, & {\rm in}~(s,+\infty)\times\Rd,\\[1mm]
u_n(s,\cdot)=f, & {\rm in}~\Rd,
\end{array}
\right.
\end{eqnarray*}
provided by Lemma \ref{lemma-2.1} satisfies the estimate
\begin{equation}
|D^3_xu_n|^p\le e^{\sigma_{3,p,n}}G^{(n)}(\cdot,s)\bigg (\sum_{j=0}^3|D^hf|^2\bigg )^{\frac{p}{2}}
\label{jack}
\end{equation}
in $(s,+\infty)\times\Rd$, where $\sigma_{3,p,n}$ is defined by \eqref{sigma-k-p} with $c_{d,3}$
being replaced by $c_{d,3,n}=7d^2(5\vee 3d^2)/(12L_n)$.
Note that $\sigma_{3,p,n}$ converges to $\sigma_{3,p}$ as $n\to +\infty$. By Lemma \ref{lem-bonaventura} we
can let $n$ tend to $+\infty$ in both the side of \eqref{jack} obtaining \eqref{aaaa}.

{\em Step 3.}
Finally, the case when $p>2$ follows easily from the case $p=2$. Indeed,
\begin{align*}
|D^k_xG(t,s)f|^p=(|D^k_xG(t,s)f|^2)^{\frac{p}{2}}\le\bigg [e^{\s_{k,2}(t-s)}G(t,s)\Big (\sum_{j=0}^k|D^jf|^2\Big )\bigg ]^{\frac{p}{2}},
\end{align*}
for $k=1,2,3$, and  we get (\ref{aaaa}) just observing that
$(G(t,s)h)^{p/2} \le G(t,s)h^{p/2}$,
for any $t>s$ and any nonnegative function $h\in C_b(\Rd)$.
\end{proof}

\begin{remark}
We stress that the condition $|c|\le \varrho^2$ in $\Rd$ is not needed to prove \eqref{aaaa} (with $h=k$) for nonnegative functions $f$.
Indeed, if $f\in C_b(\Rd)$ is nonnegative, then the function $G(t,s)f$ is strictly positive in $\Rd$ as a consequence of the strong maximum principle.
Hence, for such functions $f$, we can replace the function $w_{\tau}$ used in the proof of Theorem \ref{cucu} with the function
$w_{\tau}=\left (\sum_{j=0}^k|D^j_xu|^2\right )^{p/2}$, which is everywhere positive in $(s,+\infty)\times\Rd$.
\end{remark}

Now, we prove estimate \eqref{aa}.

\begin{hyps00}
The potential term $c$ of the operator $\A$ identically vanishes in $I\times\Rd$. Moreover,

\begin{itemize}
\item
if $k=1$, then the function $r_0+\frac{C^2d^3\nu_0^{\gamma-1}}{4(p_0-1)}\nu^{\gamma}$ is bounded from above in $I\times\Rd$ for some $p_0\in (1,+\infty)$,
where $\gamma$ is as in Hypotheses $4.1(k)$;
\item
if $k=2,3$, then there exists $M_k\in \R$ such that $r_0+L_kr\le M_k\nu^\gamma$,
where $M_k$ is any positive constant
and $L_k$ is the same constant as in Hypothesis $3.1(k)(iv)$.
\end{itemize}
\end{hyps00}

\begin{theorem}
\label{teo-sharzan}
Let Hypotheses $4.1(k)$ be satisfied, with condition $3.1(iv)$ being replaced by Hypothesis $4.2(k)$.
Then, estimate \eqref{aa} is satisfied with $\Gamma^{(1)}_{p,k}(r)=e^{\phi_{p,k}r}$
 for any $p\in (1,+\infty)$, if $k=2,3$, and for any $p \in [p_0,+\infty)$, if $k=1$, where $\phi_{p,k}$ can be explicitly computed $($see the proof\hskip 1pt$)$.
\end{theorem}

\begin{proof}
Since the proof is similar to that of Theorem \ref{cucu} we adopt here the notation therein introduced and limit ourselves to sketching it when $p \in(1,2]$. Indeed the case $p>2$ follows from the case $p=2$ and the Jensen inequality.
We begin with the case $k=3$. For any $\tau, \ve_0,\ve_1,\ve>0$ the function $w_\tau=\textstyle{(\sum_{k=1}^3|D^k u|^2+\tau)^{p/2}}$ satisfies the inequality $D_tw_{\tau}-\A w_{\tau}\le\Psi_{\tau}$ in $(s,+\infty)\times\Rd$, where
\begin{align*}
\Psi_{\tau}= p\{&{\mathscr H}_{1,p}^{(3)}(\varepsilon_0,\varepsilon_1)|\nabla_xu|^2
+{\mathscr H}_{2,p}^{(3)}(\varepsilon_0,\ve_1,\varepsilon)|D^2_xu|^2+{\mathscr H}_{3,p}^{(3)}(\ve_1,\ve)|D^3_xu|^2\notag\\
&+((1-p)\nu^{1-\gamma} +3C\ve d)\nu^\gamma|D^4_xu|^2\}w_{\tau}^{1-\frac{2}{p}}
\end{align*}
and
\begin{align}
&{\mathscr H}_{1,p}^{(3)}(\varepsilon_0,\varepsilon_1)=
\frac{Cd^2}{4\ve_0}\nu^\gamma+r_0+r\frac{d^3+d^2}{4\ve_1}
\label{sharzan-2}\\
&{\mathscr H}_{2,p}^{(3)}(\varepsilon_0,\ve_1,\varepsilon)=(1\!-\!p)\nu\!+\!\Big (C\ve_0 d\!+\!\frac{Cd^2(2+d)}{4\ve}\!+\!K\Big )\nu^\gamma\!+\!2r_0
\!+\!rd\bigg (\ve_1\!+\!\frac{3d}{4\ve_1}\bigg ),
\label{sharzan-3}\\
&{\mathscr H}_{3,p}^{(3)}(\ve,\ve_1)=(1-p)\nu+ \bigg ( 2C\ve d+\frac{3Cd^2}{4\ve}+
\ve Cd^2+3K\bigg )\nu^\gamma+3r_0+4r\ve_1d.
\label{sharzan-4}
\end{align}
Choosing $\ve_0>-Cd^2/(4M_3)$ and $\ve,\ve_1$ as in the proof of Theorem \ref{cucu} we deduce that
\begin{equation}\label{fin}
D_t w_\tau-\A w_\tau \le  p\phi_{p,3}
\big(|\nabla_xu|^2+|D^2_xu|^2+|D^3_xu|^2\big)w_{\tau}^{1-\frac{2}{p}}\le p\phi_{p,3}w_\tau+p\phi_{p,3}^{-}\tau^{p/2},
\end{equation}
where $\phi_{3,p}$ is the minimum attained by the function $(-Cd^2/(4M_3),+\infty)\ni\ve_0\mapsto \max\{{\mathscr C}_1^{(3)}(p,\ve_0),{\mathscr C}_2^{(3)}(p,\ve_0),{\mathscr C}_3^{(3)}(p)\}$
and ${\mathscr C}_1^{(3)}(p, \ve_0)=\left (\frac{Cd^2}{4\ve_0}+M_3\right )\nu_0^\gamma$,
\begin{align*}
&{\mathscr C}_2^{(3)}(p, \ve_0)=\sup_{\Rd}\bigg [(1-p)\nu+\Big(C\ve_0 d+\frac{3C^2d^3(2+d)}{4(p-1)}\nu_0^{\gamma-1}+K+2M_3\Big)\nu^{\gamma}\bigg ],\\
&{\mathscr C}_3^{(3)}(p)=\sup_{\Rd}\bigg [(1\!-\!p)\nu\!+\!\bigg (\frac{(2+d)(p-1)}{3}\nu_0^{\gamma-1}\!+\!\frac{9C^2d^3}{4(p-1)\nu_0^{\gamma-1}}\!+\!3K\!+\!3M_3\bigg )\nu^{\gamma}\bigg ].
\end{align*}
From \eqref{fin} we get \eqref{aa} with $k=3$ arguing as in the proof of Theorem \ref{cucu}.

To get the claim when $k=2$, let consider the function $w_\tau= (\sum_{k=1}^2|D^k u|^2+\tau)^{p/2}$ which satisfies the inequality
\begin{align*}
D_tw_{\tau}-\A w_{\tau}\le p\{&{\mathscr H}_{1,p}(\varepsilon_0,\varepsilon_1)^{(2)}|\nabla_xu|^2+
{\mathscr H}_{2,p}(\varepsilon_0,\varepsilon_1,\ve)^{(2)}|D^2_xu|^2\\
&+[(1-p)\nu+2C\ve d\nu^\gamma]|D^3_xu|^2\}w_{\tau}^{1-\frac{2}{p}},
\end{align*}
for any $\ve_0,\ve_1,\ve>0$,
where ${\mathscr H}_{1,p}^{(2)}(\varepsilon_0,\varepsilon_1)$ is defined as ${\mathscr H}_{1,p}^{(3)}(\varepsilon_0,\varepsilon_1)$,
with $(d^3+d^2)/(4\ve_1)$ replaced by $d^2/(4\ve_1)$, and ${\mathscr H}_{2,p}^{(2)}(\varepsilon_0,\varepsilon_1,\ve)$
is defined as ${\mathscr H}_{2,p}^{(3)}(\varepsilon_0,\varepsilon_1,\ve)$, with $Cd^2(2+d)/(4\ve)$ and $\ve_1+3d/(4\ve_1)$
replaced, respectively, by $Cd^2/(2\ve)$ and $\ve_1$.
Taking $\varepsilon=(p-1)\nu_0^{1-\gamma}/(2Cd)$, $\ve_1=\sqrt{d/2}$ and $\ve_0>-Cd^2/(4M_2)$ we get
$\psi_\tau\le p\phi_{p,2}w_\tau+p\phi_{p,2}^{-}\tau^{p/2}$,
where $\phi_{p,2}$ is the minimum attained by the function $(-Cd^2/(4M_2),+\infty)\ni\ve_0\mapsto \max\{{\mathscr C}^{(2)}_1(p,\ve_0),{\mathscr C}^{(2)}_2(p,\ve_0)\}$,
${\mathscr C}^{(2)}_1(p, \ve_0)=\left (\frac{Cd^2}{4\ve_0}+M_2\right )\nu_0^\gamma$ and
\begin{eqnarray*}
{\mathscr C}_2^{(2)}(p, \ve_0)=\sup_{\Rd}\bigg [(1-p)\nu+ \bigg (C\ve_0 d+\frac{C^2d^3}{p-1}\nu_0^{\gamma-1}+K+2M_2\bigg )\nu^\gamma\bigg ].
\end{eqnarray*}
Estimate \eqref{aa} with $k=2$ follows.

Finally, to prove \eqref{aa} with $k=1$, we consider the function $w_\tau=(|\nabla_x u|^2+\tau)^{p/2}$. In this case we get
\begin{align*}
D_tw_{\tau}-\A w_\tau\le &p\bigg\{\bigg (r_0+\frac{Cd^2}{4\ve}\nu^{\gamma}\bigg )|\nabla_xu|^2
+[(1-p)\nu+Cd\varepsilon\nu^{\gamma}]|D^2_xu|^2\bigg\}w_{\tau}^{1-\frac{2}{p}}
\end{align*}
in $(s,+\infty)\times\Rd$, for any $\ve>0$. We take $\varepsilon=(p-1)\nu_0^{1-\gamma}(Cd)^{-1}$ to get
$D_tw_{\tau}-\A w_{\tau}\le p\phi_{1,p}w_\tau$
in $(s,+\infty)\times\Rd$ for $p \ge p_0$, where
\begin{eqnarray*}
\phi_{1,p}=\sup_{I\times\Rd}\left (r_0+\frac{C^2d^3\nu_0^{\gamma-1}}{4(p-1)}\nu^{\gamma}\right ).
\end{eqnarray*}

Thus, \eqref{aa} follows also in this case.
\end{proof}

Under additional assumptions the above estimates can be proved also for $p=1$.

\begin{hyps1}
The diffusion coefficients $q_{ij}$ $(i,j=1,\ldots,d)$ are independent of $x$ and $c\equiv 0$ in $I\times\Rd$.
Moreover, $r_0+L_k'r$ is bounded from above in $I\times\Rd$, where
$L_1'=r_0$, $L_2'=(d/2)^{3/2}$, $L_3'=d\sqrt{3(d+d^2)}/3$.
\end{hyps1}

\begin{theorem}
Under Hypotheses $4.1(k)$, with Hypothesis $3.1(k)(iv)$ being replaced by Hypothesis $4.3(k)$, estimate \eqref{aa} holds true also with $p=1$.
\end{theorem}

\begin{proof}
We first consider the case $k=3$. Since the diffusion coefficients are independent of $x$, for any $\tau>0$ the function
$w_{\tau}=(|\nabla_xu|^2+|D^2_xu|^2+|D^3_xu|^2+\tau)^{1/2}$ satisfies the differential inequality
\begin{align*}
D_tw_{\tau}-\A w_{\tau}\le w_{\tau}^{-1}[{\mathscr H}_1^{(3)}(\varepsilon_1)|\nabla_xu|^2
+{\mathscr H}_2^{(3)}(\ve_1)|D^2_xu|^2+{\mathscr H}_3^{(3)}(\ve_1)|D^3_xu|^2],
\end{align*}
where
\begin{align*}
&{\mathscr H}_1^{(3)}(\varepsilon_1)=
r_0+r\frac{d^3+d^2}{4\ve_1},\qquad\;\, {\mathscr H}_2^{(3)}(\ve_1)=2r_0+rd\bigg (\ve_1+\frac{3d}{4\ve_1}\bigg ),\\
&{\mathscr H}_3^{(3)}(\ve_1)=3r_0+4r\ve_1dm,
\end{align*}
(see \eqref{sharzan-2}, \eqref{sharzan-3} and \eqref{sharzan-4}).
If we take $\ve_1=\sqrt{3(d+d^2)}/4$, then the functions ${\mathscr H}^{(3)}_j$ are all bounded in $I\times\Rd$ and estimate \eqref{aa} follows.

On the other hand, to prove \eqref{aa} with $k=2$, it suffices to observe that the function
$w_{\tau}=(|\nabla_xu|^2+|D^2_xu|^2+\tau)^{1/2}$ solves the differential equation
$D_tw_{\tau}-\A w_{\tau}\le w_{\tau}^{-1}[{\mathscr H}_1^{(2)}(\varepsilon_1)|\nabla_xu|^2
+{\mathscr H}_2^{(2)}(\ve_1)|D^2_xu|^2]$, where
${\mathscr H}^{(1)}_2(\ve_1)$ is defined as ${\mathscr H}_1^{(3)}$, with $d^3+d^2$ being replaced by $d^2$
and ${\mathscr H}^{(2)}_2(\ve_1)$ is defined as ${\mathscr H}_2^{(3)}$, with $\ve_1+3d/(4\ve_1)$ being replaced by $\ve_1$.
If we take $\ve_1=(d/2)^{1/2}$, then ${\mathscr H}^{(2)}_1(\ve_1)$ and ${\mathscr H}^{(2)}_2(\ve_1)$ are bounded in $I\times\Rd$
and \eqref{aa} follows also in this case.

Finally, the function $w_{\tau}=(|\nabla_xu|^2+\tau)^{1/2}$ satisfies the differential inequality
$D_tw_{\tau}-\A w_{\tau}\le r_0w_{\tau}$ and \eqref{aa} follows also in this case with $\phi_{1,1}=r_0$.
\end{proof}

To conclude this section we complete the proof of estimate \eqref{aaaa}. In the proof of Theorem \ref{coco} we will make use of the following result
and the following additional assumption.

\begin{hypothesis}
\label{hyp-ultima}
For any bounded interval $J\subset I$ there exists a function $\varphi_J\in C^2(\Rd)$, which blows up as $|x|\to +\infty$, such that
$\A\varphi_J\le M_J$ in $J\times\Rd$ and some positive constant $M_J$.
\end{hypothesis}

Clearly, this assumption is stronger than Hypothesis \ref{hyp-1}(iv) and it allows to prove that, if $(f_n)\in C_b(\Rd)$ is a bounded sequence
which converges locally uniformly to a function $f\in C_b(\Rd)$, as $n\to +\infty$, then
$G(\cdot,\cdot)f_n$ converges to $G(\cdot,\cdot)f$ locally uniformly in $\{(s,t)\in I\times I: s\le t\}\times\Rd$.
As a byproduct of this result, it follows that the function $(s,t,x)\mapsto (G(t,s)g)(x)$ is continuous in $\{(s,t)\in I\times I: s\le t\}\times\Rd$
for any $g\in C_b(\Rd)$. For further details, we refer the reader to \cite[Proposition 3.6, Theorem 3.7]{KunLorLun09Non}.

\begin{lemma}
\label{lemma-R-eps}
The following properties hold true.
\begin{enumerate}[\rm (i)]
\item
Let the sequence $(g_n)\in C((\sigma,t)\times\Rd)$
converge pointwise in $(\sigma,t)\times\Rd$ to a continuous function $g:(\sigma,t)\times \Rd\to\R$ satisfying $\sup_{n\in\N}\|g_n(\tau,\cdot)\|_{\infty}<+\infty$ for any $\tau\in (\sigma,t)$. Further, let $G_n(t,s)$ be the evolution operator associated with the realization in $C_b(B_{n})$ of the operator $\A$ with homogenous Dirichlet boundary conditions.
Then, $G_n(t,\cdot)g_n$ converges to $G(t,s)g$ pointwise in $(\sigma,t)\times\Rd$.
\item
Let $(f_n)\subset C_b(\Rd)$ be a bounded sequence converging to a function $f\in C_b(\Rd)$, locally uniformly in $\Rd$. If Hypothesis $\ref{hyp-ultima}$ is satisfied,
then, $G(t,s+1/n)f_n$ converges to $G(t,s)f$, locally uniformly in $\Rd$, for any $t>s$.
\end{enumerate}
\end{lemma}

\begin{proof} (i) Since $|G_n(t,\tau)f|\le G_n(t,\tau)|f|\le G(t,\tau)|f|$ for any function $f\in C_b(\Rd)$ (where we have used the fact that for any nonnegative function $g\in C_b(\Rd)$, $G_n(t,\tau)g$
pointwise increases to $G(t,\tau)g$ as $n\to +\infty$), we can estimate
\begin{align}
&|G_n(t,\tau)g_n(\tau,\cdot)-G(t,\tau)g(\tau,\cdot)|\notag\\
\le & G(t,\tau)|g_n(\tau,\cdot)- g(\tau,\cdot)|+|G(t,\tau)g(\tau,\cdot)-G_n(t,\tau)g(\tau,\cdot)|,
\label{star:star}
\end{align}
for any $m\in\N$.
Using the representation formula \eqref{kernel}, it is easy to check that
$G(t,\tau)|g_n(\tau,\cdot)- g(\tau,\cdot)|$ pointwise converges to zero as $n$ tends to $+\infty$.
On the other hand, the second term in the last side \eqref{star:star} vanishes as $n\to +\infty$ as it has been already remarked above.
Hence, the assertion follows.

(ii) Note that
\begin{align*}
|G(t,s+1/n)f_n - G(t,s)f|\le &|G(t,s+1/n)f_n-G(t,s+1/n)f|\\
&+|G(t,s+1/n)f-G(t,s)f|,
\end{align*}
for any $n\in\N$. Since $G(\cdot,\cdot)f_n$ converges to $G(\cdot,\cdot)f$ locally uniformly in
$\{(s,t)\in I\times I: s\le t\}\times\Rd$, as recalled above, the first term in the right-hand side of the previous inequality vanishes locally uniformly in $\Rd$.
Also the second term vanishes locally uniformly in $\Rd$ since the function $G(\cdot,\cdot)f$ is continuous in $\{(s,t)\in I\times I: s\le t\}\times\Rd$.
\end{proof}

\begin{theorem}\label{coco}
Let Hypotheses $4.1(k)$ be satisfied for some $k\in\{1,2,3\}$, with the functions in Hypothesis $3.1(k)(ii)$ being replaced by two positive constants $C_1$ and $C_2$.
Further, assume that also Hypothesis \ref{hyp-ultima} is satisfied. Then, estimate \eqref{aaaa} is satisfied
for some positive function $\Gamma_{p,k,k-1}$, which can be explicitly computed $($see the proof\hskip 2pt$)$.
\end{theorem}

\begin{proof} We limit ourselves to proving estimate \eqref{aaaa} when $k=3$, considering first the case $p\in (1,2]$.
Without loss of generality we can assume that $c_0\le 0$. Indeed, if this is not the case, it suffices
to replace the evolution operator $G(t,s)$ with the evolution operator $e^{-c_0(t-s)}G(t,s)$.
We also assume that $c$ is bounded, since if $c$ is unbounded then it can be approximated by
a sequence of smooth functions which satisfy Hypothesis 4.1(3) (see the proof of Theorem \ref{cucu} for further details).

To simplify the notation, throughout the remaining of the proof, we set $u_m=G_m(\cdot,s)f$ and
$u=G(\cdot,s)f$, where, as usual, $G_m(t,s)$ is the evolution operator associated in $C(\overline B_m)$
with the operator $\A$ with homogeneous Dirichlet boundary conditions, and $f\in C^2_b(\Rd)$. Moreover, for any $m\in\N$, we set $\vartheta_m(x)=\vartheta(m^{-1}|x|)$, where
$\vartheta$ is a smooth cut-off function such that $\chi_{[0,1/2]}\le\vartheta\le\chi_{[0,1]}$. Finally, all the integrals
that we consider are to be understood pointwise, i.e., given $g\in C([a,b]\times\Rd)$, by $\int_a^bg(s,\cdot)ds$
we mean the function $x\mapsto\int_a^bg(s,x)ds$ for any $x\in\Rd$.

For any $\alpha,\beta>0$, $t>s$, $\delta\in (0,1)$, $m\in\N$,  and $f\in C^2_b(\R^d)$
we define the function $g_{\delta}:[s,t]\to C(\overline{B_m})$ by setting
$g_{\delta}(\tau,\cdot)=G_m(t,\tau)(\psi_m(\tau,\cdot)-\delta^{p/2})$ for any $\tau\in [s,t]$, where
$\psi_m=(\alpha|u_m|^2+\beta\vartheta^2_m|\nabla_xu_m|^2+\vartheta_m^4|D^2_xu_m|^2+\delta)^{p/2}$.
Since $\psi_m-\delta$ vanishes on $\partial B_m$, taking \cite[Theorem 2.3(ix)]{acqui} into account, we can show that the function $g_{\delta}$
is differentiable in $(s,t)$ and
$g'_{\delta}=G_m(t,\cdot)(D_{\tau}\psi_m-\A(\tau)\psi_m+c(\tau,\cdot)\delta^{p/2})$, where
\begin{align*}
D_{\tau}\psi_m-\A\psi_m=&p\psi_m^{1-\frac{2}{p}}\sum_{i=1}^3{\mathscr J}_i(u)+
\frac{p(2\!-\!p)}{4}\langle Q\xi_{u_m},\xi_{u_m}\rangle\psi_m^{1-\frac{4}{p}}\\
&+(p-1)c\psi_m-p\delta c\psi_m^{1-\frac{2}{p}},
\end{align*}
\begin{align*}
&{\mathscr J}_1(u_m)=-\alpha {\mathscr Q}_0(u_m)-\beta\vartheta_m^2\si{\mathscr Q}_0(D_iu_m)-\vartheta_m^4\sij{\mathscr Q}_0(D_{ij}u_m)\\
&\phantom{{\mathscr J}_1(u_m)=}
+\beta\vartheta_m^2{\mathscr B}_1(u_m)+2\vartheta_m^4\si {\mathscr B}_1(D_iu_m),\\[1mm]
&{\mathscr J}_2(u_m)=
\beta\vartheta_m^2{\mathscr Q}_1(u_m)+2\vartheta_m^4\si{\mathscr Q}_1(D_iu_m)+\vartheta_m^4{\mathscr Q}_2(u_m)\\
&\phantom{{\mathscr J}_2(u_m)=}
-4\beta\vartheta_m {\mathscr Q}_3(u_m)-8\vartheta_m^3\si {\mathscr Q}_3(D_iu_m)+\vartheta_m^4{\mathscr B}_2(u_m),\\[1mm]
&{\mathscr J}_3(u_m)=\beta\vartheta_m^2{\mathscr C}_1(u_m)+2\vartheta_m^4\si{\mathscr C}_1(D_iu_m)
+\vartheta_m^4{\mathscr C}_2(u_m),\\[1mm]
&{\mathscr J}_4(u_m)=-\frac{\beta}{2}(\mathcal{A}_0\vartheta^2_m)|\nabla_xu_m|^2
-\frac{1}{2}(\mathcal{A}_0\vartheta_m^4)|D^2_xu_m|^2,
\end{align*}
$\xi_{u_m}=\nabla_x(\alpha |u_m|^2+\beta\vartheta_m^2|\nabla_xu_m|^2+\vartheta_m^4|D^2_xu_m|^2)$, $\A_0=(\A-c)$, ${\mathscr Q}_i$ ($i=0,1,2$), ${\mathscr B}_j$ ($j=1,2$) are defined
at the beginning of the proof of Theorem \ref{cucu} and
${\mathscr Q}_3(\zeta)=\langle Q\nabla\vartheta_m,D^2\zeta\nabla\zeta\rangle$
for any smooth enough function $\zeta$.

Arguing as in the proof of Theorem \ref{cucu} and using the inequality
$(a+b)^2\le (1+\ve)a^2+(1+\ve^{-1})b^2$, which holds true for any $a,b,\ve>0$, we deduce that
\begin{align*}
&\langle Q\xi_{u_m},\xi_{u_m}\rangle\\
\le &(1+\ve)\bigg [\alpha|u_m|({\mathscr Q}_0(u_m))^{\frac{1}{2}}
+\beta\vartheta_m^2|\nabla_xu_m|\bigg (\si{\mathscr Q}_0(D_iu_m)\bigg )^{\frac{1}{2}}\\
&\hs{5}\qq\qq\ds+\vartheta_m^4|D^2_xu_m|\bigg (\sij {\mathscr Q}_0(D_{ij}u_m)\bigg )^{\frac{1}{2}}
\bigg ]^2\\
&+\frac{1+\ve}{\ve}(\beta\vartheta_m|\nabla_xu_m|^2
+2\vartheta_m^3|D^2_xu_m|^2)^2{\mathscr Q}_0(\vartheta_m)\\
\le& (1+\ve)\bigg [\alpha {\mathscr Q}_0(u_m)+\beta\vartheta_m^2
\si {\mathscr Q}_0(D_iu_m)+\vartheta_m^4\sij {\mathscr Q}_0(D_{ij}u_m)\bigg ]\psi_m^{\frac{2}{p}}\\
&\hs{7}\;\ds+\frac{1+\ve}{\ve}(\beta |\nabla_xu_m|^2+
4\vartheta_m^2|D^2_xu_m|^2)(\beta\vartheta_m^2 |\nabla_xu_m|^2+\vartheta_m^4|D^2_xu_m|^2){\mathscr Q}_0(\vartheta_m).
\end{align*}
Further, taking  Hypothesis 4.1(3)(ii) and the choice of $\vartheta_m$ into account it can be easily checked that
\begin{align*}
[{\mathscr Q}_0(\vartheta_m)](t,x)=&|\vartheta'(m^{-1}|x|)|^2m^{-2}|x|^{-2}\langle Q(t,x)x,x\rangle\notag\\
\le &C_1\|\vartheta'\|_{\infty}^2(1+|x|^2)|x|^{-1}m^{-2}\nu(t,x)\notag\\
\le &4\|\vartheta'\|_{\infty}^2C_1m^{-3}(1+m^2)\nu(t,x)=:\frac{K_1}{m}\nu(t,x),
\end{align*}
for any $(t,x)\in I\times\Rd$ and
$\mathcal{A}_0(\vartheta^2_m) \geq -K_2\nu$,
\begin{align*}
&\mathcal{A}_0(\vartheta^4_m)= 2 \vartheta_m^2 \mathcal{A}_0(\vartheta^2_m) +2{\mathscr Q}_0(\vartheta_m^2)
\geq 2 \vartheta_m^2 \mathcal{A}_0(\vartheta^2_m)\geq -2K_2\vartheta_m^2\nu,\\
&|\vartheta_m {\mathscr Q}_3(u_m)|\leq\frac{K_3}{4\ve}\nu|\nabla_xu_m|^2+
\ve K_3\vartheta_m^2\nu |D^2_xu_m|^2,\\
&\bigg |\vartheta^3_m\si {\mathscr Q}_3(D_iu_m)\Bigg |
\leq \frac{K_3}{4\ve}\vartheta_m^2\nu|D^2_xu_m|^2
+\ve K_3\vartheta_m^4\nu|D^3_xu_m|^2,
\end{align*}
in $I\times\Rd$, for any $\ve>0$, where $K_2$ and $K_3$ are positive constants independent of $u$ and $\ve$.

Taking all the previous estimates, the fact that $c$ is bounded and nonpositive and \eqref{estim-Q1}, \eqref{estim-Q2} and \eqref{estim-B2}  into account, we conclude that
\begin{align*}
D_{\tau}\psi_m-\A\psi_m
\le\bigg\{&\frac{\beta+d}{4\ve}u_m^2+
{\mathscr K}_{1,\ve,p}|\nabla_xu_m|^2+{\mathscr K}_{2,\ve,p}\vartheta_m^2|D^2_xu_m|^2\\
&+\{[(2\!-\!p)(1\!+\!\ve)\!-\!1\!+\!8\ve K_3]\nu-2C\ve d\nu^{\gamma}\}\vartheta_m^4|D^3_xu_m|^2\bigg\}\psi_m^{1-\frac{2}{p}}\\
&+p\|c\|_{\infty}\delta\psi_m^{1-\frac{2}{p}},
\end{align*}
where
\begin{align*}
{\mathscr K}_{1,\ve,p}=&\bigg [[(2-p)(1+\ve)-1]\alpha+\frac{\beta K_2}{2}
+\frac{\beta K_3}{\ve}+(2-p)\frac{\beta K_1}{m}\frac{1+\ve}{\ve}\bigg ]\nu\\
&+\beta \frac{Cd^2}{4\ve}\nu^{\gamma}+\beta\vartheta_m^2\bigg (r_0+r\frac{d^2}{4\beta\ve}+d\varepsilon\rho^2\bigg )
+\frac{1}{2\ve},\\[1mm]
{\mathscr K}_{2,\ve,p}=&\bigg [[(2-p)(1+\ve)-1+4\ve K_3]\beta+K_2+\frac{2K_3}{\ve}
+4(2-p)\frac{K_1}{m}\bigg (1+\frac{1}{\ve}\bigg )\bigg ]\nu\\
&+\bigg (\beta Cd\ve +\frac{Cd^2}{2\ve}+K\bigg )\nu^{\gamma}+\vartheta_m^2(2r_0+r\ve d+3d\ve\rho^2).
\end{align*}
Thanks to Hypothesis 3.1(3)(iv) (where, now, $L$ and $M$ are constants), we can fix $\ve$ sufficiently small
such that $(2-p)(1+\ve)-1+8\ve K_3<0$ and $2r_0+r\ve d+3d\ve\rho^2<2M\nu^{\gamma}$ in $I\times\Rd$.
Next, we fix $\beta>1$ such that the supremum over $I\times\Rd$ of ${\mathscr K}_{2,\ve,p}$
is negative and
$r_0+rd^2(4\beta\ve)^{-1}+d\varepsilon\rho^2\le M\nu^{\gamma}$
in $I\times\Rd$. Finally, we choose $\alpha>\beta$ large enough to
make negative the supremum over $I\times\Rd$ of ${\mathscr K}_{1,\ve,p}$.
As a byproduct, taking into account that we are assuming that $c$ is bounded,
we can determine two positive constants $K_{1,p}$ and $K_{2,p}$ such that
\begin{align*}
D_{\tau}\psi_m-\A\psi_m
\le& -K_{1,p}(|\nabla_xu_m|^2+|D^2_xu_m|^2+|D^3_xu_m|^2)\vartheta_m^4\psi_m^{1-\frac{2}{p}}\\
&+K_{2,p}u_m^2\psi_m^{1-\frac{2}{p}}+p\|c\|_{\infty}\delta\psi_m^{1-\frac{2}{p}}
\end{align*}
and, thus,
\begin{align*}
g'_{\delta}\le &-K_{1,p}G_m(t,\cdot)[(|\nabla_xu_m|^2+|D^2_xu_m|^2+|D^3_xu_m|^2)\vartheta_m^4\psi_m^{1-\frac{2}{p}}]\\
&+K_{2,p}G_m(t,\cdot)\psi_m+p\|c\|_{\infty}\delta G_m(t,\cdot)\psi_m^{1-\frac{2}{p}}.
\end{align*}
Integrating this inequality over $[s+\ve,t-\ve]$ (for $\ve\in (0,(t-s)/2)$), pointwise in $\Rd$, and observing that $G_m(t,\cdot)\psi_m^{1-2/p}\le \delta^{p/2-1}G(t,\cdot)\mathds{1}
\le \delta^{p/2-1}$ yield
\begin{align}
&K_{1,p}\int_{s+\ve}^{t-\ve}G_m(t,\tau)[(|\nabla_xu_m(\tau,\cdot)|^2+|D^2_xu_m(\tau,\cdot)|^2+|D^3_xu_m(\tau,\cdot)|^2)\vartheta_m^4\psi_m^{1-\frac{2}{p}}]d\tau\notag\\
\le & G_m(t,s+\ve)\psi_m(s+\ve,\cdot)+K_{2,p}\int_{s+\ve}^{t-\ve}G_m(t,\tau)\psi_m(\tau,\cdot)d\tau+p\|c\|_{\infty}\delta^{\frac{p}{2}}(t-s).\notag\\
\label{CC}
\end{align}
\vskip -4truemm

Next, using Lemma \ref{lemma-R-eps} and the dominated convergence theorem, we let first $m$ tend to $+\infty$ and then $\ve$ tend to $0^+$ to get\footnote{We stress that the proof of the uniform
estimates in Theorem \ref{C0-C3}, given in \cite{Lorenzi-1}, shows that the function in square brackets in \eqref{CC} can be estimated from above by a constant, independent of $m$, times $(\tau-s)^{-1/2}$.}
\begin{align*}
&K_{1,p}\int_s^tG(t,\tau)[(|\nabla_xu(\tau,\cdot)|^2+|D^2_xu(\tau,\cdot)|^2+|D^3_xu(\tau,\cdot)|^2)\psi^{1-\frac{2}{p}}]d\tau\\
\le & G(t,s)\psi(s,\cdot)+K_{2,p}\int_s^tG(t,\tau)\psi(\tau,\cdot)d\tau+p\|c\|_{\infty}\delta^{\frac{p}{2}}(t-s),
\end{align*}
where $\psi$ is defined as $\psi_m$, with $u_m$ being replaced by $u$.
Now, using estimate \eqref{aaaa} with $h=k=3$, splitting $u(t,\cdot)=G(t,\tau)u(\tau,\cdot)$,
from Young and H\"older inequalities, \eqref{fg} and \eqref{fp}, which shows that $G(t,\tau)|u(\tau,\cdot)|^p\le G(t,s)|f|^p$, we deduce that
\begin{align}
&e^{-\s_{3,p}(t-\tau)}(|\nabla_xu(t,\cdot)|^2+|D^2_xu(t,\cdot)|^2+|D^3_xu(t,\cdot)|^2)^{\frac{p}{2}}\notag\\
\le &G(t,\tau)\Big [(|\nabla_xu(\tau,\cdot)|^2+|D^2_xu(\tau,\cdot)|^2+|D^3_xu(\tau,\cdot)|^2)^{\frac{p}{2}}\Big ]+G(t,\tau)|u(\tau,\cdot)|^p\notag\\
\le & G(t,\tau)\left [(|\nabla_xu(\tau,\cdot)|^2+|D^2_xu(\tau,\cdot)|^2+|D^3_xu(\tau,\cdot)|^2)^{\frac{p}{2}}\psi^{\frac{p}{2}-1}\psi^{1-\frac{p}{2}}\right ]
+G(t,s)|f|^p\notag\\
\le& \Big\{G(t,\tau)\Big [(|\nabla_xu(\tau,\cdot)|^2\!+\!|D^2_xu(\tau,\cdot)|^2\!+\!|D^3_xu(\tau,\cdot)|^2)\psi^{1-\frac{2}{p}}\Big ]\Big\}^{\frac{p}{2}}(G(t,\tau)\psi)^{1-\frac{p}{2}}\notag\\
&+G(t,s)|f|^p\notag\\
\le &\frac{p}{2}\varepsilon^{\frac{2}{p}}G(t,\tau)\Big [(|\nabla_xu(\tau,\cdot)|^2\!+\!|D^2_xu(\tau,\cdot)|^2\!+\!|D^3_xu(\tau,\cdot)|^2)\psi^{1-\frac{2}{p}}\Big ]\notag\\
&+\frac{2-p}{2}\varepsilon^{\frac{2}{p-2}}G(t,\tau)\psi(\tau,\cdot)+G(t,s)|f|^p\notag\\
\label{AA}
\end{align}
\vskip -3truemm
\noindent
for any $\ve>0$. From now on, we assume that both the constants $\sigma_{2,p}$ and $\sigma_{3,p}$ do not vanish.

Since $\alpha>\beta>1$, arguing as in the proof of Theorem \ref{cucu} it can be shown that
\begin{align}
G(t,\tau)\psi(\tau,\cdot)\le &G(t,\tau)\{e^{\sigma_{2,p}(\tau-s)}\alpha^{\frac{p}{2}}G(\tau,s)[(f^2+|\nabla f|^2+|D^2f|^2+\delta)^{\frac{p}{2}}]\}\notag\\
=&\alpha^{\frac{p}{2}}e^{\sigma_{2,p}(\tau-s)}G(t,s)[(f^2+|\nabla f|^2+|D^2f|^2+\delta)^{\frac{p}{2}}]
\label{AB}
\end{align}
for any $\tau\in [s,t]$ and, consequently,
\begin{align}
\int_s^tG(t,\tau)\psi(\tau,\cdot)d\tau\le \alpha^{\frac{p}{2}}\frac{e^{\sigma_{2,p}(t-s)}-1}{\sigma_{2,p}}G(t,s)[(f^2+|\nabla f|^2+|D^2f|^2+\delta)^{\frac{p}{2}}].
\label{AC}
\end{align}
From \eqref{AA} and \eqref{AB} we get
\begin{align*}
&e^{-\s_{3,p}(t-\tau)}(|\nabla_xu(t,\cdot)|^2+|D^2_xu(t,\cdot)|^2+|D^3_xu(t,\cdot)|^2)^{\frac{p}{2}}\\
\le &\frac{p}{2}\varepsilon^{\frac{2}{p}}G(t,\tau)\Big [(|\nabla_xu(\tau,\cdot)|^2\!+\!|D^2_xu(\tau,\cdot)|^2\!+\!|D^3_xu(\tau,\cdot)|^2)\psi^{1-\frac{2}{p}}\Big ]\\
&+\frac{2-p}{2}\varepsilon^{\frac{2}{p-2}}\alpha^{\frac{p}{2}}e^{\sigma_{2,p}(\tau-s)}G(t,s)[(f^2+|\nabla f|^2+|D^2f|^2+\delta)^{\frac{p}{2}}]
+G(t,s)|f|^p.
\end{align*}
Integrating this inequality in $(s,t)$ and taking \eqref{CC} and \eqref{AC} into account, we get
\begin{align}
&\frac{1-e^{-\s_{3,p}(t-s)}}{\s_{3,p}}(|\nabla_xG(t,s)f|^2+|D^2_xG(t,s)f|^2
+|D^3_xG(t,s)f|^2)^{\frac{p}{2}}\notag\\
\le &{\mathscr H}_{p,\ve}(t-s)G(t,s)[(f^2+|\nabla f|^2+|D^2f|^2+\delta)^{\frac{p}{2}}]\notag\\
&+\frac{p^2}{2K_{1,p}}\ve^{\frac{2}{p}}\|c\|_{\infty}\delta^{\frac{p}{2}}(t-s)
+(t-s)G(t,s)|f|^p+\frac{2-p}{2}\ve^{\frac{2}{p-2}}\alpha^{\frac{p}{2}}\delta^{\frac{p}{2}},
\label{stimaconde}
\end{align}
where
\begin{align*}
{\mathscr H}_{p,\ve}(r)=\frac{p}{2K_1}\ve^{\frac{2}{p}}\alpha^{\frac{p}{2}}\bigg (1+K_{2,p}\frac{e^{\sigma_{2,p}r}-1}{\sigma_{2,p}}\bigg )
+\frac{2-p}{2}\ve^{\frac{2}{p-2}}\alpha^{\frac{p}{2}}\frac{e^{\sigma_{2,p}r}-1}{\sigma_{2,p}}.
\end{align*}
Letting $\delta$ tend to $0^+$ in (\ref{stimaconde}) and estimating $G(t,s)|f|^p\le G(t,s)(f^2+|\nabla f|^2+|D^2f|^2)^{p/2}$ yield
\begin{align*}
&\frac{1-e^{-\s_{3,p}(t-s)}}{\s_{3,p}}\bigg (\sum_{h=1}^3|D^h_xG(t,s)f|^2\bigg )^{\frac{p}{2}}\\
\le &[{\mathscr H}_{p,\ve}(t-s)+(t-s)]G(t,s)\bigg (\sum_{h=0}^2|D^hf|^2\bigg )^{\frac{p}{2}}.
\end{align*}

Estimate \eqref{aaaa} follows, with
\begin{align*}
\Gamma_{p,2,3}(r)=\frac{\s_{p,3}}{1-e^{-\s_{p,3}r}}\bigg\{\bigg [\frac{\alpha}{K_1}\bigg (1+K_{p,2}\frac{e^{\sigma_{p,2}r}-1}{\sigma_{p,2}}\bigg )\bigg ]^{\frac{p}{2}}
\bigg (\frac{e^{\sigma_{p,2}r}-1}{\sigma_{p,2}}\bigg )^{1-\frac{p}{2}}+r\bigg\},
\end{align*}
by minimizing over $\ve>0$. The previous formula holds also in the case when at least one between $\sigma_{p,2}$ and $\sigma_{p,3}$ vanishes,
provided one replaces the ratio $(e^{\sigma r}-1)/\sigma$ by $r$.

To obtain \eqref{aaaa} for $p>2$ it suffices to write
$|D^3_xG(t,s)f|^p=(|D^3_xG(t,s)f|^2)^{\frac{p}{2}}$, apply \eqref{aaaa} with $p=2$ and then use \eqref{fp}.
\end{proof}

\begin{corollary}
Under the hypotheses of Theorem $\ref{coco}$, estimate \eqref{aaaa} holds true for any $h<k-1$.
\end{corollary}

\begin{proof}
The proof follows from applying repeatedly estimate \eqref{aaaa} with $h=k-1$. For the reader's convenience, we provide the proof in the case $k=3$.
For this purpose, we fix $t>s\in I$, $p\in (1,+\infty)$, set $t_1=(t+s)/2$ and observe that
\begin{align*}
|D^3_xG(t,s)f|^p=&|D^3_xG(t,t_1)G(t_1,s)f|^p\\
\le &\Gamma_{p,2,3}^{(2)}\bigg (\frac{t-s}{2}\bigg )
G(t,t_1)\bigg (\sum_{j=0}^2|D_x^jG(t_1,s)f|^2\bigg )^{\frac{p}{2}}\\
\le &\Gamma_{p,2,3}^{(2)}\bigg (\frac{t-s}{2}\bigg )\Gamma_{p,1,2}^{(2)}\bigg (\frac{t-s}{2}\bigg )
G(t,t_1)G(t_1,s)\bigg (\sum_{j=0}^1|D^jf|^2\bigg )^{\frac{p}{2}}\\
=&\Gamma_{p,3}^{(2)}\bigg (\frac{t-s}{2}\bigg )\Gamma_{p,1,2}^{(2)}\bigg (\frac{t-s}{2}\bigg )
G(t,s)\bigg (\sum_{j=0}^1|D^jf|^2\bigg )^{\frac{p}{2}}
\end{align*}
and \eqref{aaaa} follows with $h=1$, $k=3$ and $\Gamma^{(2)}_{p,1,3}(r)=\Gamma_{p,2,3}^{(2)}(r/2)\Gamma_{p,1,2}^{(2)}(r/2)$.

Finally, to prove \eqref{aaaa}, with $h=0$ and $k=3$, we fix $t>s\in I$, $p\in (1,+\infty)$, set $t_2=s+(t+s)/3$ and observe that
 \begin{align*}
|D^3_xG(t,s)f|^p=&|D^3_xG(t,t_2)G(t_2,s)f|^p\\
\le&\Gamma_{p,2,3}^{(2)}\bigg (\frac{t-s}{3}\bigg )
G(t,t_2)\bigg (\sum_{j=0}^2|D_x^jG(t_2,s)f|^2\bigg )^{\frac{p}{2}}\\
=&\Gamma_{p,2,3}^{(2)}\bigg (\frac{t-s}{3}\bigg )
G(t,t_2)\bigg (\sum_{j=0}^2|D_x^jG(t_2,t_1)G(t_1,s)f|^2\bigg )^{\frac{p}{2}}\\
=&\Gamma_{p,2,3}^{(2)}\bigg (\frac{t-s}{3}\bigg )\Gamma_{p,1,2}\bigg (\frac{t-s}{3}\bigg )
G(t,t_1)\bigg (\sum_{j=0}^1|D_x^jG(t_1,s)f|^2\bigg )^{\frac{p}{2}}\\
\le &\Gamma_{p,2,3}^{(1)}\bigg (\frac{t-s}{3}\bigg )\Gamma_{p,1,2}\bigg (\frac{t-s}{3}\bigg )
\Gamma_{p,0,1}\bigg (\frac{t-s}{3}\bigg )G(t,s)|f|^p
\end{align*}
and \eqref{aaaa} follows with $\Gamma^{(2)}_{p,0,3}(r)=\Gamma_{p,2,3}^{(1)}(r/3)\Gamma_{p,1,2}^{(2)}(r/3)\Gamma_{p,0,1}^{(2)}(r/3)$.
\end{proof}

\begin{remark}
Even if in this paper we confine ourselves to the study of the operator scalar operator
$\A$ with coefficients defined in the whole $I\times\Rd$, we mention that elliptic operators with unbounded
coefficients have been considered also in unbounded domains. In this case,
uniform and pointwise estimates for the derivatives of the solution of associated Cauchy problem, with Dirichlet, Neumann
and more general homogeneous boundary conditions, have been proved (see e.g., \cite{AngLorOnD,AngLorNon,BerFor04,BerForLor,BerForLor07,cerrai,HieLorRha,KryPri10Ell}).
Also in some situations where the elliptic operator \eqref{oper-A} is replaced by
a system of elliptic operators some gradient estimates are available (see e.g, \cite{AddAngLorTes,HieLorPruRha,LorDel11}).
\end{remark}

\section{Some consequences of the pointwise estimates of Section \ref{sect-4}}
\label{sect-5}
If not otherwise specified, throughout this section we assume that $c\equiv 0$.

The following example, due to J. Pr\"uss, A. Rhandi and R. Schnaubelt,
 shows that the $L^p$-spaces related to the Lebesgue measure do not represent a \emph{good setting}
where to study the evolution operator $G(t,s)$.

\begin{example}[Section 2 of \cite{pruss-rhandi-schnaubelt}]
\label{ex-PRS}
{\rm Let $\A$ be the one-dimensional elliptic operator
defined by $(\A\psi)(x)=\psi''(x)-x|x|^{\varepsilon}\psi'(x)$ for any $x\in\R$
and smooth enough functions $\psi:\R\to\R$. Let us show that for any $p\in [1,+\infty)$, $\lambda>0$ and any nonnegative and not identically
vanishing functions $f\in C^{\infty}_c(\R)$ the equation $\lambda u-\A u=f$ does not admit solutions in $L^p(\R)$.
Indeed, fix $p$, $\lambda$ and $f$ as above and, by contradiction let assume that the equation $\lambda u-\A u=f$ admits a solution $u\in L^p(\R)$. By elliptic regularity $u$ belongs to $C^2(\R)$. Let us prove that $u$ is bounded in $\R$. For this purpose,
we take advantage of the one-dimensional Feller theory (see \cite{feller}) as in Remark \ref{rem-2.1}(iii). In this case, the function
$Q$ and $R$ are given by
\begin{eqnarray*}
Q(x)=e^{-\frac{|x|^{\ve+2}}{\ve+2}}\int_0^xe^{\frac{|t|^{\ve+2}}{\ve+2}}dt,\qquad\;\,
R(x)=e^{\frac{|x|^{\ve+2}}{\ve+2}}\int_0^xe^{-\frac{|t|^{\ve+2}}{\ve+2}}dt,\q\;\, x\in\R.
\end{eqnarray*}
A straightforward computation reveals that $\lim_{x\to\pm\infty}x^{1+\ve/2}Q(x)=0$. Hence, $Q\in L^1(\R)$. On the other hand, the function $R$
does not belong to $L^1(\R)$. This implies that the equation
$\lambda u-\A u=0$ admits a decreasing solution $u_1$ which tends to $1$ at $+\infty$ and an increasing solution $u_2$ which tends to $1$ at $-\infty$ (see \cite[Chapter 2]{LorLibro-2}).
Clearly, $u_1$ and $u_2$ are linearly independent. Moreover $u_1$ and $u_2$ diverge to $+\infty$ at $-\infty$ and at $+\infty$, respectively.
Since the equation $\lambda u-\A u=f$ admits a bounded solution $u_b\in C_b(\R)\cap C^2(\R)$ and any other solution is given by
$c_1u_1+c_2u_2+u_b$ for some real constants $c_1$ and $c_2$, if $(c_1,c_2)\neq (0,0)$ then the function
$u_b+c_1u_1+c_2u_2$ does not belong to $L^p(\R)$. Consequently $u=u_b$, i.e., $u$ is bounded and positive since $u_b$ is.

We now introduce the functions $V:\R\to\R$ and $W:\R\setminus\{0\}\to\R$, defined by $V(x)=x^2$ and $W(x)=\ve(2\lambda)^{-1}+|x|^{-\ve}$
for $x\in\R$ and $x\neq 0$, respectively, which satisfy the differential inequalities $\lambda V-\A V\ge 0$ and $\lambda W-\A W\le 0$ in $\R\setminus (-r,r)$ if $r$ is properly chosen.
Since $u$ is bounded and positive, we can fix $\beta>0$ such that $u(\pm r)\ge\beta W(\pm r)$.
For any $\delta>0$ the function $w=u-\beta W+\delta V$ satisfies the differential inequality
$\lambda w-\A w\ge 0$ in $\R\setminus (-r,r)$ and diverges to $+\infty$ as $x$ tends to $\pm\infty$.
Hence, it admits a nonnegative minimum in $\R\setminus (-r,r)$. Thus,
$u-\beta W+\delta V\ge 0$ in $\R\setminus (-r,r)$ for any $\delta>0$ and letting $\delta$ tend to $0^+$ yields $u\ge\beta W$. Since $W\notin L^p(\R)$
the function $u$ does not belong to $L^p(\R)$ as well and the contradiction follows.
}\end{example}

\begin{remark}\label{lppres}
Sufficient conditions for the evolution $G(t,s)$ to preserve $L^p(\Rd)$ are obtained in \cite{AngLor10Com} and, in the case when $c\equiv 0$ (as we are assuming in this section), they require, besides
Hypothesis \ref{hyp-1}(i)-(ii), that
one of the following set of conditions is satisfied:
\begin{enumerate}[\rm (a)]
\item
the coefficients $q_{ij}$ and $b_j$ ($i,j=1,\ldots,d$) are differentiable in $I\times\Rd$ with respect to the spatial variables, the weak derivatives $D_{ij}q_{ij}$ exist in $I\times\Rd$
for any $i,j$ as above and the function $\beta:I\times\Rd\to\Rd$ defined by
$\beta_i=b_i-\sj D_{ij}q_{ij}$ for any $i=1,\ldots,d$ is such that ${\rm div}_x\beta \ge- K_0$ in $I\times\Rd$ for some positive constant $K_0$;
\item
the coefficients $q_{ij}$ are differentiable in $I\times\Rd$ with respect to the spatial variables, the function $\nu(t,\cdot)$ in Hypothesis \ref{hyp-1}(i) is measurable for any $t\in I$,
and the function $|\beta|^2$ is controlled from above by a constant $K_1$ times the function $\nu$.
\end{enumerate}
\end{remark}
In the first case, $G(t,s)$ preserves $L^p(\Rd)$ for any $t>s\in I$ and $p\in [1,+\infty)$, whereas
the second set of conditions guarantee that $L^p(\Rd)$ is preserved by the action of the evolution operator if $p\in (1,+\infty)$.

On the other hand, the evolution operator $G(t,s)$ enjoys good properties in the $L^p$-spaces related
to the \emph{tight}$^(\footnote{A set of Borel measures $\{\mu_t:\;t\in I\}$ in $\Rd$ is tight if for every $\varepsilon >0$ there exists $\rho>0$ such that $\mu_t(\Rd\setminus B_\rho)\leq \varepsilon$, for every $t\in I$. }^)$ evolution system of measures. The existence of such measures can be proved under the following set of hypotheses, which we assume throughout this section, if not otherwise mentioned.

\begin{hypothesis}
\label{hyp-2}
Hypotheses $\ref{hyp-1}(i)$-$(ii)$ are satisfied and there exists a nonnegative function $\varphi \in C^2(\Rd)$, which blows up as $|x|\to+\infty$, such that
$\A\varphi\le a_1 -a_2\varphi$ in $[t_0,+\infty)\times \Rd$ for some positive constants $a_1,a_2$ and $t_0\in I$.
\end{hypothesis}

Under Hypothesis \ref{hyp-2} in \cite{KunLorLun09Non} it has been proved that there exists a tight evolution system of measures $\{\mu_t: t\in I\}$ associated to the evolution operator $G(t,s)$.
The invariance property \eqref{invariance} and formula \eqref{fp} show that
\begin{eqnarray*}
\int_{\Rd}|G(t,s)f|^pd\mu_t\le\int_{\Rd}G(t,s)|f|^pd\mu_t=\int_{\Rd}|f|^pd\mu_s
\end{eqnarray*}
for any $f\in C_b(\Rd)$, $p\in (1,+\infty)$ and $I\ni s<t$. The density of $C^{\infty}_b(\Rd)$ (and, hence, $C_b(\Rd)$) in
$L^p(\Rd,\mu_s)$, allows us to extend each operator $G(t,s)$ with a contraction from $L^p(\Rd,\mu_s)$ into $L^p(\Rd,\mu_t)$.

\begin{remark}
Note that the operator considered in Example \ref{ex-PRS} satisfies the previous hypothesis.
Indeed, if one take as $\varphi$ the function defined by $\varphi(x)=x^2$ for any $x\in\R$, one easily
realizes that $\A\varphi(x)=2-2x^{2+\ve}$ for any $x\in\R$ and it is easy to check that
there exist two positive constants $a_1$ and $a_2$ such that $\A\varphi\le a_1-a_2\varphi$ in $\R$.
Hence, Hypothesis \ref{hyp-2} are weaker than those in \cite{AngLor10Com}, which guarantee that $G(t,s)$ preserves $L^p(\Rd)$.
\end{remark}

In this section, we illustrate remarkable properties enjoyed by the evolution operator $G(t,s)$ in the $L^p$-spaces related to evolution systems of measures, which are consequences
of the pointwise estimates of the previous section. To begin with, we consider the following result.

\begin{proposition}
\label{prop-wilde}
Under the same hypotheses as in Theorems $\ref{cucu}$ and $\ref{coco}$ the following assertions hold true.
\begin{enumerate}[\rm(i)]
\item
If the operator $G(t,s)$ is bounded\,\footnote{which is the case under conditions (a) in Remark \ref{lppres}} in $L^1(\Rd)$ for any $t>s \in I$, then it is bounded from $W^{\theta_1,p}(\Rd)$ to
$W^{\theta_2,p}(\Rd)$,
for any $0\le
\theta_1\le \theta_2\le k$, $p\in (1,+\infty)$, $t>s\in I$ and
$\|G(t,s)\|_{L(W^{\theta_1,p}(\Rd),W^{\theta_2,p}(\Rd))}\le C_{\theta_1,\theta_2,p}(s,t)$
for any $\theta_1,\theta_2,p,k,t$ and $s$ as above and some positive function $C_{\theta_1,\theta_2,p}$, explicitly determined in the proof;
\item
each operator $G(t,s)$ is bounded from $W^{\theta_1,p}(\Rd,\mu_s)$ to $W^{\theta_2,p}(\Rd;\mu_t)$,
for any $p\in (1,+\infty)$, $\theta_1,\theta_2\in\{0,\ldots,k\}$, with $\theta_1\le\theta_2$, $t>s\in I$ and
\begin{eqnarray*}
\|G(t,s)\|_{L(W^{\theta_1,p}(\Rd,\mu_s),W^{\theta_2,p}(\Rd,\mu_t))}\le \widetilde C_{\theta_1,\theta_2,p}(t-s),\qquad\;\,t>s\in I,
\end{eqnarray*}
\end{enumerate}
for any $p,\theta_1,\theta_2$ as above and some positive function $\widetilde {C}_{\theta_1,\theta_2,p}:(0,+\infty)\to (0,+\infty)$
which can be explicitly computed $($see the proof\hskip 2pt$)$.
\end{proposition}

\begin{proof}
(i) We set $c_{\alpha,\beta,q}(s,t)=\|G(t,s)\|_{L(W^{\alpha,q}(\Rd),W^{\beta,q}(\Rd))}$ for any $I\ni s<t$,
$0\le\alpha\le \beta$ and $q\in [1,+\infty)$.
By assumptions $c_{0,0,1}(s,t)$ is finite for any $s$ and $t$ as above.
By interpolation, from \eqref{reality} we deduce that also $c_{0,0,p}(s,t)$ is finite for any $t>s\in I$, $p\in (1,+\infty)$ and
$c_{0,0,p}(s,t)\le (c_{0,0,1}(s,t))^{1/p}$.

Now, fix $k$ as in the statement of Theorem \ref{coco}. Integrating estimate \eqref{aaaa} (with $h=0$) in $\Rd$ with
respect to the Lebesgue measure and using the density of $C^\infty_c(\Rd)$ in $W^{k,p}(\Rd)$, we get
\begin{align*}
c_{0,k,p}(s,t)\le |c_{0,0,1}(s,t)|^{\frac{1}{p}}\bigg [\sum_{j=1}^k(\Gamma^{(2)}_{p,0,j}(t-s))^{\frac{1}{p}}+1\bigg ]
\end{align*}
for any $t>s \in I$, $p\in (1,+\infty)$. Moreover, even from \eqref{aaaa} it follows that
\begin{align*}
&\|D^h_xG(t,s)f\|_{L^p(\Rd)}\\
\le &(\Gamma^{(2)}_{p,h,h}(t-s))^{\frac{1}{p}}(c_{0,0,1}(t-s))^{\frac{1}{p}}
\bigg (\int_{\Rd}\bigg (\sum_{j=0}^h|D^jf|^2\bigg )^{\frac{p}{2}}dx\bigg )^{\frac{1}{p}}\\
\le &(\Gamma^{(2)}_{p,h,h}(t-s))^{\frac{1}{p}}(c_{0,0,1}(t-s))^{\frac{1}{p}}\max\big\{2^{\frac{1}{2}-\frac{1}{p}},1\big\}\bigg (\int_{\Rd}\sum_{j=0}^h |D^jf|^pdx\bigg )^{\frac{1}{p}}\\
\le &(\Gamma^{(2)}_{p,h,h}(t-s))^{\frac{1}{p}}(c_{0,0,1}(t-s))^{\frac{1}{p}}\max\big\{2^{\frac{1}{2}-\frac{1}{p}},1\big\}\|f\|_{W^{h,p}(\Rd)}
\end{align*}
for any $h\in\{1,\ldots,k\}$, $t>s\in I$, $p\in (1,+\infty)$.
Hence,
\begin{align*}
c_{k,k,p}(s,t)\le |c_{0,0,1}(s,t)|^{\frac{1}{p}}\max\big\{2^{\frac{1}{2}-\frac{1}{p}},1\big\}\bigg [\sum_{h=1}^k(\Gamma^{(2)}_{p,h,h}(t-s))^{\frac{1}{p}}+1\bigg ].
\end{align*}
The claim is thus proved for $(\theta_1,\theta_2)=(0,k)$ and $(\theta_1,\theta_2)=(k,k)$. The
remaining cases follow by interpolation, taking into account that for any
$\theta \in (0,1)$ and $p \in [1, +\infty)$, $W^{(1-\theta)\theta_1+\theta\theta_2,p}(\Rd)=(W^{\theta_1,p}(\Rd),W^{\theta_2,p}(\Rd))_{\theta,p}$ with equivalence of the respective norms (see \cite[Theorem 2.4.1(a)]{triebel}).
More precisely, since $G(t,s)$ belongs to $L(L^p(\Rd), W^{k,p}(\Rd))\cap L(W^{k,p}(\Rd))$, it follows that
$G(t,s)\in L(W^{\theta_1,p}(\Rd), W^{k,p}(\Rd))$ for any $\theta_1\in (0,k)$ and
$$c_{\theta_1,k,p}(s,t)\le (c_{0,k,p}(s,t))^{1-\theta_1/k}(c_{k,k,p}(s,t))^{\theta_1/k}$$ for any $I\ni s<t$.
Moreover, since $G(t,s)\in L(L^p(\Rd))\cap L(W^{k,p}(\Rd))$ for any $t>s\in I$, $G(t,s)$ is bounded from $W^{\theta_1,p}(\Rd)$ into itself and
\begin{eqnarray*}
c_{\theta_1,\theta_1,p}(s,t)\le (c_{0,0,p}(s,t))^{1-\frac{\theta_1}{k}}(c_{k,k,p}(s,t))^{\frac{\theta_1}{k}},\qquad\;\,I\ni s<t.
\end{eqnarray*}
Finally, using the fact that $G(t,s)\in L(W^{\theta_1,p}(\Rd))\cap L(W^{\theta_1,p}(\Rd),W^{k,p}(\Rd))$, we conclude that
$G(t,s)\in L(W^{\theta_1,p}(\Rd),W^{\theta_2,p}(\Rd))$ for any $0\le\theta_1\le\theta_2\le k$ and
$c_{\theta_1,\theta_2,p}(s,t)\le (c_{\theta_1,\theta_1,p}(s,t))^{(k-\theta_2)/(k-\theta_1)}(c_{\theta_1,k,p}(s,t))^{(\theta_2-\theta_1)/(k-\theta_1)}$
for any $t>s\in I$. The claim follows.

(ii) The proof is obtained immediately integrating the pointwise estimates \eqref{aaaa},
taking the invariance property of the evolution system $\{\mu_t: t\in I\}$ into account and arguing as in (i).
We get $\widetilde{C}_{0,0,p}(r)=1$,
\begin{align*}
&\widetilde{C}_{h,h,p}(r)=\max\big\{2^{\frac{1}{2}-\frac{1}{p}},1\big\}\sum_{j=1}^h(\Gamma^{(2)}_{p,j,j}(r))^{\frac{1}{p}}+1,\qquad\;\,h\ge 1;\\[1mm]
&\widetilde{C}_{h,k,p}(r)=\max\big\{2^{\frac{1}{2}-\frac{1}{p}},1\big\}^h\sum_{j=0}^k(\Gamma^{(2)}_{p,h,j}(r))^{\frac{1}{p}}+1,\qquad\;\,h=0,1;\\[1mm]
&\widetilde{C}_{2,3,p}(r)=C_{1,1,p}(r)+\max\big\{2^{\frac{1}{2}-\frac{1}{p}},1\big\}\sum_{j=2}^3(\Gamma^{(2)}_{p,2,2}(r))^{\frac{1}{p}}+1.\qedhere
\end{align*}
\end{proof}
\begin{remark}
{\rm Under conditions (a) in Remark \ref{lppres}, the functions $c_{0,0,p}(s,t)$, $p \ge 1$ in the proof of
 Proposition \ref{prop-wilde}, are explicit. More precisely $c_{0,0,p}(s,t)=e^{K_0(t-s)/p}$.}
\end{remark}

Whenever the uniform estimate
\begin{equation}
\|\nabla_xG(t,s)f\|_{\infty}\le e^{\sigma_{\infty}(t-s)}\|f\|_{C^1_b(\Rd)},\qquad\;\,t>s\in I,\;\,f\in C^1_b(\Rd),
\label{around}
\end{equation}
holds true for some negative constant $\sigma_{\infty}$, the tight evolution system of measures is unique, as the following Proposition \ref{norma_p} shows.
Note that Theorem \ref{cucu} provides us with a sufficient condition for \eqref{around} to hold.

In the rest of this section we denote by $\overline{f}_s$ the average of $f$ with respect to the measure $\mu_s$, i.e.,
\begin{eqnarray*}
\overline{f}_s=\int_{\Rd}fd\mu_s,\qquad\;\, s\in I.
\end{eqnarray*}
\begin{proposition}
\label{norma_p}
If \eqref{around} holds true, then the tight evolution system of measures associated to $G(t,s)$ is unique.
\end{proposition}

\begin{proof}
We fix $s\in I$, $f\in C^1_b(\Rd)$, set $r_t=e^{-\sigma_\infty t/2}$, for any  $t\geq s$, and observe that,
for any $t>s$ and $x\in\Rd$,
\begin{align*}
|(G(t,s)f)(x)-\overline f_s|=&\bigg |\int_{\Rd}[(G(t,s)f)(x)-G(t,s)f]d\mu_t\bigg |\nonumber\\
\le &\int_{B_{r_t}} |(G(t,s)f)(x)-G(t,s)f|d\mu_t\nonumber\\
&+ \int_{\Rd\setminus B_{r_t}} |(G(t,s)f)(x)-G(t,s)f|d\mu_t\nonumber\\
\leq & \Vert \nabla_x G(t,s)f\Vert_\infty \int_{B_{r_t}} |x-y|d\mu_t+2\|f\|_{\infty}\mu_t(\Rd\setminus B_{r_t})\nonumber\\
\leq & e^{\sigma_{\infty}(t-s)}\|\nabla f\|_\infty \bigg (|x|+ \int_{B_{r_t}}|y|d\mu_t\bigg )+ 2\|f\|_\infty \mu_t(\Rd\setminus B_{r_t}).
\end{align*}
Hence,
$\|G(t,s)f-\overline f_s\|_{C_b(B_R)}\le e^{\sigma_{\infty}(t-s)}R\|\nabla f\|_\infty+H(s,t,f)$ for any $t>s$, where
$H(s,t,f):=e^{(\sigma_{\infty}t-2s)/2}\|\nabla f\|_\infty+2\|f\|_\infty \mu_t(\Rd\setminus B_{r_t})$.
The tightness of the measures $\{\mu_t: t\in I\}$ shows that $\mu_t(\Rd\setminus B_{r_t})$ tends to $0$ as $t\to +\infty$ and, consequently,
$\|G(t,s)f-\overline f_s\|_{C_b(B_R)}$ vanishes as $t\to +\infty$ for any $R>0$.

Using this result, we can conclude the proof. Indeed,
assume by contradiction that there exists another tight evolution system of measures $\{\nu_s: s\in I\}$ associated to $G(t,s)$. Then,
for any $f\in C^\infty_c(\Rd)$, the average of $f$ with respect to $\mu_s$ and to $\nu_s$ coincide for every $s$.
Since the characteristic function of a Borel set $A$ is the almost everywhere limit
of a sequence of functions in $C^{\infty}_c(\Rd)$, by the dominated convergence theorem, we conclude that
$\mu_s(A)=\nu_s(A)$ for every $s\in I$ and, thus, the two evolution systems of measures actually coincide.
\end{proof}

\begin{remark}
\label{rem-luca}
By the results in \cite{BKR} each measure $\mu_t$ is absolutely continuous with respect to the Lebesgue measure. More precisely, there exists a continuous function
$\rho:I\times\Rd\to\R$ such that $d\mu_t=\rho(t,\cdot)dx$.
\end{remark}

\subsection{Logarithmic Sobolev inequalities and summability improving properties}

The so-called {\it logarithmic Sobolev inequality} \eqref{Log_Sob-1} is crucial in the study of the evolution operator $G(t,s)$
in the $L^p$-spaces related to the tight evolution system of measures $\{\mu_t: t\in I\}$. These estimates, proved firstly in 1975
by Gross for the Gaussian measures, represent the counterpart of the Sobolev embedding theorems which fail in general when
the measure is not the Lebesgue measure.

\begin{example}
\label{ex-noSob}
{\rm Let $\mu$ be the one-dimensional Gaussian measure, whose density with respect to the Lebesgue measure is the function
$\psi:\R\to\R$, defined by $\psi(x)=\frac{1}{\sqrt{2\pi}}e^{-x^2/2}$ for any $x\in\R$.
For any $\ve>0$, the function $f_{\ve}(x)=e^{x^2/(4\ve)}$ belongs to $W^{1,2}(\R,\mu)$ but it does not belong to
$L^{2+\ve/2}(\R)$. Hence, no embeddings of $W^{1,2}(\R,\mu)$ into $L^q(\R,\mu)$ exist if $q>2$.
We note that $\mu$ is the invariant measure of the semigroup associated with the Ornstein-Uhlenbeck operator
$\A=D_xx-xD_x$.}
\end{example}

The logarithmic Sobolev inequality \eqref{Log_Sob-1} yields some relevant results as the next proposition shows.

\begin{proposition}
Assume that \eqref{Log_Sob-1} is satisfied. Then, the following assertions hold true.
\begin{enumerate}[\rm (i)]
\item
$W^{1,p}(\Rd,\mu_s)$ is compactly embedded in $L^p(\R^d,\mu_s)$ for any $p\in [2,+\infty)$ and $s\in I$;
\item
for any $t>s$ and $p\in (1,+\infty)$, $G(t,s)$ is a compact operator from $L^p(\Rd,\mu_s)$ into $L^p(\Rd, \mu_t)$;
\item
The Poincar\'e inequality
$\|f- \overline{f}_s\|_{L^2(\R^d,\mu_s)}\leq 2^{-1}C_2\|\nabla f\|_{L^2(\R^d,\mu_s)}$
holds true for any $f\in W^{1,2}(\Rd,\mu_s)$ and $s\in I$.
\end{enumerate}
\end{proposition}

\begin{proof}
(i) Fix $p\ge 2$. The logarithmic Sobolev inequality implies that, for any $s\in I$, the $\|f\|_{L^p(\Rd\setminus B_R,\mu_s)}$ vanishes as $R\to +\infty$, uniformly with respect to
$f$ in the closed unit ball of $W^{1,p}(\Rd,\mu_s)$.
Indeed,  for any $f\in W^{1,p}(\Rd,\mu_s)$ and $k\in\N$, introduce the set $E_k=\{x\in\Rd: |f(x)|\le k\}$ and observe that the logarithmic Sobolev inequality \eqref{Log_Sob-1} (which can be extended by density to any function in $W^{1,p}(\Rd,\mu_s)$) and H\"older inequality show that
\begin{align*}
\|f\|_{L^p(\Rd\setminus B_R,\mu_s)}^p=&\int_{E_k\cap (\Rd\setminus B_R)}|f|^pd\mu_s
+\int_{\Rd\setminus (B_R\cup E_k)}|f|^pd\mu_s\\
\le & k^p\mu_s(\Rd\setminus B_R)
+\frac{1}{\log(k)}\int_{\Rd}|f|^p\log(|f|)d\mu_s\\
\le & k^p\mu_s(\Rd\setminus B_R)\\
&+\frac{1}{\log(k)}\bigg [\|f\|_{L^p(\Rd,\mu_s)}^p\!\log(\|f\|_{L^p(\Rd,\mu_s)})\!+\!\frac{C_p}{p}\|f\|_{W^{1,p}(\Rd,\mu_s)}\bigg ].
\end{align*}
Hence, $\|f\|_{L^p(\Rd\setminus B_R,\mu_s)}\le k^p\mu_s(\Rd\setminus B_R)+M(\log(k))^{-1}$
for any $R,k>0$ and some positive constant $M$,
if $\|f\|_{W^{1,p}(\Rd,\mu_s)}\le 1$. Letting first $R$ and, then, $k$ tend to $+\infty$, the claim follows.

To conclude the proof, it suffices to show that, for any $R>0$,
the set $\{f_{|B_R}: f\in W^{1,p}(\Rd,\mu_s),~\|f\|_{W^{1,p}(\Rd,\mu_s)}\le 1\}$ is totally bounded in
$L^p(B_R,\mu_s)$, but this follows straightforwardly, from observing that the measure $\mu_s$ is absolutely continuous with respect to the Lebesgue measure and
its density is a positive continuous functions. This shows that $L^p(B_R)=L^p(B_R,\mu_s)$, with equivalence of the corresponding norms, and the Rellich-Kondrakov
theorem shows that $W^{1,p}(B_R)$ is compactly embedded into $L^p(B_R)$.

(ii) The proof follows from (i), if $p\ge 2$, recalling that each operator $G(t,s)$ is bounded from $L^p(\Rd,\mu_s)$ into $W^{1,p}(\Rd,\mu_s)$ (see Proposition \ref{prop-wilde})
To prove it for $p\in (1,2)$ it suffices to apply Stein interpolation theorem (see \cite[Thm. 1.6.1]{Davies}) taking into account that $G(t,s)$ is bounded from $L^1(\Rd,\mu_s)$ into $L^1(\Rd,\mu_t)$,
 for any $t>s\in I$.

(iii)
By the density of $C^1_b(\Rd)$ in $W^{1,2}(\Rd,\mu_s)$, it suffices to prove the Poincar\'e inequality for functions in $C^1_b(\Rd)$. Moreover it is not restrictive to assume that $\overline{f}_s=0$.
Indeed, once the Poincar\'e inequality is proved for functions with zero average with respect to $\mu_s$, applying it to the function $f-\overline f_s$, we get it in the general case.

The proof of the Poincar\'e inequality for functions $f\in C^1_b(\Rd)$ with $\overline f_s=0$ follows from applying the logarithmic Sobolev inequality (with $p=2$) to
the function $1+\varepsilon f$ ($\ve>0$), then dividing both sides by $\ve$ and letting $\ve\to 0^+$.
\end{proof}

\begin{remark}
\label{rem-Poi}
{\rm
The Poincar\'e  inequality can be proved also for $p>2$ and some positive constant $\widetilde C_p$. A classical proof can be found for example in \cite[Theorem 5.8.1]{evans} and is based on the compact embedding of $W^{1,p}(\Rd,\mu_s)$ into $L^p(\Rd,\mu_s)$ for $p \ge 2$. On the other hand, another approach relies on an iterative procedure which starts from the case $p=2$. Differently from the first approach, the second one, adopted in \cite{AngLorLun}, allows to control how $\widetilde C_p$ depends on $s$.
}\end{remark}

A sufficient condition for the logarithmic Sobolev inequality to hold is proved in \cite{AngLorLun}.
The main tool of the proof is the pointwise gradient estimate
\begin{equation}\label{stimapoint}
|(\nabla_xG(t,s)f)(x)|\le e^{\sigma_1(t-s)}(G(t,s)|\nabla f|)(x),\qquad\;\, t>s,\;\, x \in \Rd,\;\,f\in C_b^1(\Rd),
\end{equation}
for some $\sigma_1<0$. Whenever \eqref{stimapoint} holds, estimate \eqref{around} is satisfied with $\sigma_{\infty}=\sigma_1$. Hence, there exists
a unique tight evolution system of measures. In the rest of the section, we always deal with such an evolution system of measures.

\begin{theorem}[Theorem 3.3 of \cite{AngLorLun}]
\label{thm_log_sob}
Suppose that the diffusion coefficients of the operator $\A$ are independent of $x$ and bounded. Further, suppose that
$\langle {\rm Jac}_xb(t,x)\xi,\xi\rangle\le r_0|\xi|^2$ for any $t\in I$, $x, \xi \in \Rd$ and some negative constant $r_0$.
Then, estimate \eqref{stimapoint} holds true for any $f\in W^{1,p}(\Rd,\mu_s)$, $p\in (1,+\infty)$, $s\in I$, with $\sigma_1=r_0$, estimate \eqref{Log_Sob-1} holds true with $C_p=(2|r_0|)^{-1}p^2\Lambda_0$,
where $\Lambda_0$ denotes the supremum over $I$ of the maximum eigenvalue of the matrix $Q(t)$.
\end{theorem}

Under the assumptions of Theorem \ref{thm_log_sob}, {\it which we assume as standing assumptions in the rest of this subsection}, it can be proved a first summability improving result of the evolution operator $G(t,s)$.

\begin{theorem}[Theorem 4.1 of \cite{AngLorLun}]
\label{th-5.2}
Under the assumptions of Theorem $\ref{thm_log_sob}$
the evolution operator $G(t,s)$ is hypercontractive, i.e.,
for any $p,q\in (1,+\infty)$, with $p<q$, the operator $G(t,s)$ is a contraction from $L^p(\Rd,\mu_s)$ into $L^q(\Rd,\mu_t)$ if $t \ge s+\frac{\Lambda_0}{2\nu_0|r_0|}\log\left (\frac{q-1}{p-1}\right )$.
\end{theorem}

It is also interesting to study some stronger summability improving properties of the evolution operator $G(t,s)$.
These stronger summability improving properties are:
\begin{itemize}
\item
{\it supercontractivity:} $G(t,s)$ is bounded from $L^p(\Rd,\mu_s)$ into $L^q(\Rd,\mu_t)$ for any $q>p>1$ and $t>s\in I$;
\item
{\it ultraboundedness:} $G(t,s)$ is bounded from $L^p(\Rd,\mu_s)$ into $C_b(\Rd)$ for any $p\in (1,+\infty)$ and $t>s\in I$;
\item
{\it ultracontractivity:} $G(t,s)$ is bounded from $L^1(\Rd,\mu_s)$ into $C_b(\Rd)$ for any $t>s\in I$;
\end{itemize}

The following theorem shows that the supercontractivity is equivalent to the occurrence of a one-parameter family of logarithmic Sobolev inequalities
and to an integrability property of the Gaussian functions $\varphi_\lambda:\Rd\to\R$, defined by $\varphi_\lambda(x):=e^{\lambda |x|^2}$ for any $x \in \Rd$ and $\lambda>0$ with respect to the measures $\mu_s$ ($s\in I$).

\begin{theorem}[Theorems 3.1 \& 3.7  of \cite{AngLorOnI}]
The following facts are equivalent.
\begin{enumerate}[\rm(i)]
\item
The evolution operator $G(t,s)$ is supercontractive.
\item
The inequality
\begin{align*}
&\int_{\Rd}|f|^p\log(|f|)d\mu_s-\|f\|_{L^p(\Rd,\mu_s)}^p\log(\|f\|_{L^p(\Rd,\mu_s)})\\
\le &\varepsilon\frac{p}{2}\int_{\Rd}|f|^{p-2}|\nabla f|^2d\mu_s+\frac{2\beta(\varepsilon)}{p}\|f\|^p_{L^p(\Rd,\mu_s)}
\end{align*}
holds true for every $f\in W^{1,p}(\Rd, \mu_s)$, $s\in I$, $p>1$, $\varepsilon>0$ and some positive decreasing function $\beta:(0,+\infty)\to (0,+\infty)$ which blows up as $\varepsilon\to 0^+$.
\item
The function $\varphi_\lambda$ belongs to $L^1(\Rd, \mu_s)$ for any $\lambda>0$ and $s\in I$. Moreover, $\sup\{\Vert \varphi_{\lambda}\Vert_{L^1(\Rd, \mu_s)}:s\in I\}<+\infty$ for any $\lambda>0$.
\end{enumerate}
\end{theorem}

On the other hand, the ultraboundedness can be characterized as follows.

\begin{theorem}[Theorem 4.5 of \cite{AngLorOnI}]
The evolution operator $G(t,s)$ is ultrabounded, if and only if
for every $\lambda>0$ and $t>s$ the function $G(t,s)\varphi_{\lambda}$ belongs to $C_b(\Rd)$ and, for any $\delta, \lambda>0$, there exists a positive constant $K_{\delta, \lambda}$ such that $\|G(t,s)\varphi_\lambda\|_{\infty}\le K_{\delta, \lambda}$ for any $t>s\in I$.
\end{theorem}

\begin{remark}{\rm
\begin{enumerate}[\rm(i)]
\item
A sufficient condition for the supercontractivity of the evolution operator $G(t,s)$ is the existence of a positive constant $K$ such that $\langle b(t,x),x\rangle\le  -K|x|^2\log|x|$ for any $t \in I$ and $x$ large enough. This condition is quite sharper. Indeed, the autonomous operator
$(\mathcal{A}\zeta)(x)=\Delta \zeta(x)-\langle x,\nabla \zeta(x)\rangle$ does not satisfy it
and it is well known that the associated
Ornstein-Uhlenbeck semigroup is not supercontractive with respect to the Gaussian invariant
measure $d\mu(x)=(2\pi)^{-d/2}e^{-|x|^2/2}dx$ as proved in \cite{Nelson}.
\item
In order to prove that the evolution operator $G(t,s)$ is ultrabounded it suffices to assume that there exist $K_1>0$ and $\alpha>1$ such that $\langle b(t,x),x\rangle\le  -K_1|x|^2(\log|x|)^\alpha$ for any $t \in I$ and $x$ large enough.
Also this condition is rather sharp. Indeed in \cite{KavKerRoy93Que}, the authors show that
the semigroup associated with the operator
$\mathcal{A}=\Delta-\langle \nabla\Phi,\nabla \rangle$ is not ultrabounded in the $L^p$-spaces related to the invariant measure
$d\mu=\|e^{-\Phi}\|^{-1}_{1}e^{-\Phi}dx$, if $\Phi(x)\sim |x|^2\log|x|$.
\end{enumerate}
}\end{remark}

An equivalent characterization of the ultracontractivity is not available in the literature, at the best of our knowledge. On the other hand a sufficient condition is given
by the following theorem.

\begin{theorem}
Suppose that $\langle b(t,x),x\rangle \le -K_2|x|^\gamma$ for any $t \in I$, $|x|\ge R$ and some positive constants $K_2$, $R$ and $\gamma>2$.
Then, the evolution operator $G(t,s)$ is ultracontractive.
\end{theorem}

\subsection{Long-time behaviour of $G(t,s)f$}

This last subsection is devoted to present some result on the asymptotic behaviour of $G(t,s)f$ as $t\to +\infty$.
As the proof of Proposition \ref{norma_p} shows,  $G(t,s)f$ converges to $\overline f_s$ locally uniformly in $\Rd$ as $t\to +\infty$ for any $f\in C^1_b(\Rd)$,
provided that the gradient estimate \eqref{around} is satisfied. In such a case, one can also infer that
$\|G(t,s)f-\overline f_s\|_{L^p(\Rd,\mu_s)}$ vanishes as $t\to +\infty$ for any $f\in L^p(\Rd,\mu_s)$ and $s\in I$, using the above local uniform convergence, the density
of $C^1_b(\Rd)$ into $L^p(\Rd,\mu_s)$ and the uniform boundedness (with respect to $s$ and $t$ of $\|G(t,s)\|_{L(L^p(\Rd,\mu_s),L^p(\Rd,\mu_t))}$
and of the operator $f\mapsto \overline{f}_s$ from $L^p(\Rd,\mu_s)$ into $\R$.

Actually we can be more precise on the decay rate to $0$ of the previous norm when some additional conditions are satisfied.
For any $p\in [1,+\infty)$, we introduce the sets  $\mathfrak{A}_p$ and $\mathfrak{B}_p$ defined as follows:
\begin{itemize}
\item
$\mathfrak{A}_p$ is the set of all $\omega\in\R$ such that
\begin{align*}
\|G(t,s)f-m_s(f)\|_{L^p(\Rd,\mu_t)}\le M_{p,\omega}e^{\omega (t-s)}\|f\|_{L^p(\Rd,\mu_s)}
\end{align*}
for any $f\in L^p(\Rd,\mu_s)$, any $I\ni s<t$ and some positive constant $M_{p,\omega}$;
\item
$\mathfrak{B}_p$ is the set of all $\omega\in\R$ such that
\begin{align*}
\|\nabla_x G(t,s)f\|_{L^p(\Rd,\mu_t)}\le N_{p,\omega}e^{\omega(t-s)}\|f\|_{L^p(\Rd,\mu_s)}
\end{align*}
for any $f\in L^p(\Rd,\mu_s)$, any $I\ni s<t$, such that $t-s\ge 1$, and some positive constant $N_{p,\omega}$.
\end{itemize}

\begin{theorem}[Theorem 5.3 of \cite{AngLorLun}]\label{Asy}
The following facts are true:
\begin{enumerate}[\rm (i)]
\item
suppose that $G(t,s)$ is bounded from $L^{p_0}(\Rd,\mu_s)$ into $W^{1,p_0}(\Rd,\mu_t)$ and
$$\|\nabla_xG(t,s)f\|_{L^{p_0}(\Rd,\mu_t)}\le C_1(t-s)\|f\|_{L^{p_0}(\Rd,\mu_s)}$$ for any $f\in L^{p_0}(\Rd,\mu_s)$, $t>s\in I$,
some $p_0\in (1,+\infty)$ and a positive function $C_1:(0,+\infty)\to (0,+\infty)$. Then, $\mathfrak{A}_{p_0}\subset {\mathfrak B}_{p_0}$.
\item
if the evolution operator $G(t,s)$ is hypercontractive, then the sets $\mathfrak{A}_p$ is independent of $p\in (1,+\infty)$;
\item
if the evolution operator $G(t,s)$ is hypercontractive and
$G(t,s)$ is bounded from $L^p(\Rd,\mu_s)$ into $W^{1,p}(\Rd,\mu_t)$ and
$$\|\nabla_xG(t,s)f\|_{L^p(\Rd,\mu_t)}\le C_2(t-s)\|G(t,s)|\nabla f|\|_{L^p(\Rd,\mu_s)}$$ for\footnote{In view of Theorem \ref{teo-sharzan},
this condition is satisfied if $r_0\nu^{-\gamma}$ diverges to $-\infty$ as $|x|\to +\infty$, uniformly with respect to $t\in I$.}
any $f\in L^p(\Rd,\mu_s)$, any $p\in (1,+\infty)$ and
some positive and locally bounded function $C_2:(0,+\infty)\to (0,+\infty)$, then the sets $\mathfrak{B}_p$ are independent of $p\in (1,+\infty)$;
\item
if the assumptions in $(i)$ are satisfied for any $p\in (1,+\infty)$ as well as the assumptions in $(iii)$ and, in addition, the Poincar\'e inequality holds true,
then $\mathfrak{A}_p=\mathfrak{B}_p$, for any $p\in (1,+\infty)$.
\end{enumerate}
\end{theorem}

\begin{proof}
(i) Let $p_0$ be as in the statement, fix $\omega\in \mathfrak{A}_{p_0}$,
$s,t\in I$ with $t-s\ge 1$, and $f\in L^{p_0}(\R^d,\mu_s)$ with $\overline f_s=0$.
Splitting $G(t,s)f=G(t,t-1)G(t-1,s)f$ and using the estimate in the statement, we get
\begin{align*}
\|\nabla_x G(t,s)f\|_{L^{p_0}(\R^d,\mu_t)}=&\|\nabla_x G(t,t-1)G(t-1,s)f\|_{L^{p_0}(\R^d,\mu_t)}\\
\le & C_1(1)\|G(t-1,s)f\|_{L^{p_0}(\R^d,\mu_{t-1})}\\
\le & C_1(1)M_{p_0,\omega}e^{\omega(t-s)}\|f\|_{L^{p_0}(\R^d,\mu_s)}.
\end{align*}
If $\overline f_s\neq 0$, the previous estimate follows with $C_1(1)$ being replaced by $2C_1(1)$, just applying
the above estimate to  $f-\overline f_s$ and noting that
$\|f-\overline f_s\|_{L^{p_0}(\Rd,\mu_s)}\le 2\|f\|_{L^{p_0}(\Rd,\mu_s)}$.
Hence, $\omega\in\mathfrak{B}_{p_0}$, so that $\mathfrak{A}_{p_0}\subset\mathfrak{B}_{p_0}$.

(ii) We fix $p_1,p_2\in (1,+\infty)$, such that  $1<p_1<p_2$, and $\omega\in {\mathfrak A}_{p_1}$.
Moreover, we take $\tau>0$ such that $p_2=e^{2\eta_0|r_0|\Lambda^{-1}\tau}(p_1-1)+1$. If $t>\tau+s$, then, from
Theorem \ref{th-5.2} it follows that $G(t,t-\tau)$ is a contraction from $L^{p_1}(\Rd,\mu_{t-\tau})$ to
$L^{p_2}(\Rd,\mu_t)$. Thus, using the evolution law, the hypercontractivity of the evolution operator and recalling that $G(t-\tau,s)\mathds{1}=\mathds{1}$, we get
\begin{align*}
\|G(t,s)f-\overline{f}_s\|_{L^{p_2}(\Rd,\mu_t)}=&
\|G(t,t-\tau)(G(t-\tau,s)f-\overline{f}_s)\|_{L^{p_2}(\Rd,\mu_t)}\\
\le &\|G(t-\tau,s)f-\overline{f}_s\|_{L^{p_1}(\Rd,\mu_{t-\tau})}\\
\le & M_{p_1,\omega}e^{-\omega\tau}e^{\omega(t-s)}\|f\|_{L^{p_1}(\Rd,\mu_s)}
\end{align*}
for any $f\in L^{p_2}(\Rd,\mu_s)\subset L^{p_1}(\Rd,\mu_s)$, $t>s+\tau$ and some positive constant $M_{p_1,\omega}$, independent of $f$, where $\tau$ is as above. Hence, we get
$\|G(t,s)f-\overline{f}_s\|_{L^{p_2}(\Rd,\mu_t)}
\le M_{p_1,\omega}e^{-\omega \tau}e^{\omega(t-s)}\|f\|_{L^{p_2}(\Rd,\mu_s)}$
for any $f\in L^{p_2}(\Rd,\mu_s)$ and any $t> s+\tau$. This inequality can be extended to any $t\in (s,s+\tau)$, up to possibly changing the
constant $M_{p_1,\omega}$, recalling that $G(t,s)$ is a contraction from $L^{p_2}(\Rd,\mu_s)$ into $L^{p_2}(\Rd,\mu_t)$ and $|\overline{f}_s|\le \|f\|_{L^{p_2}(\Rd,\mu_s)}$ for any $t>s\in I$ .

Viceversa, fix $\omega\in \mathfrak{A}_{p_2}$ and $f\in L^{p_1}(\Rd,\mu_s)$. By the definition of the evolution systems of measures it follows easily that $\overline{(G(r,s)f)}_r=\overline f_s$ for any $r>s$.
Moreover, since $\|\cdot\|_{L^{p_1}(\Rd,\mu_s)}\le \|\cdot\|_{L^{p_2}(\Rd,\mu_s)}$, using the hypercontractivity of the evolution operator $G(t,s)$, we can estimate
\begin{align*}
&\|G(t,s)f-\overline f_s\|_{L^{p_1}(\Rd,\mu_t)}\\
\le&\|G(t,s+\tau)G(s+\tau,s)f-\overline{(G(s+\tau,s)f)}_{s+\tau}\|_{L^{p_2}(\Rd,\mu_t)}\\
\le &M_{p_2,\omega}e^{\omega(t-s-\tau)}\|G(s+\tau,s)f\|_{L^{p_2}(\Rd,\mu_{s+\tau})}\\
\le &M_{p_2,\omega}e^{\omega(t-s-\tau)}\|f\|_{L^{p_1}(\Rd,\mu_s)},
\end{align*}
for some positive constant $M_{p_2,\omega}$ and any $t > s+\tau$.
As above, this is enough to infer that $\omega\in\mathfrak{A}_{p_1}$. Summing up, we have proved that
$\mathfrak{A}_{p_1}= \mathfrak{A}_{p_2}$ for any $1<p_1<p_2<+\infty$ and, consequently, that $\mathfrak{A}_{p}$ is independent of $p\in (1,+\infty)$.

(iii) Fix $1<p_1<p_2<+\infty$, $\omega \in \mathfrak{B}_{p_1}$ and $t\geq s+ \tau+1$,
where $\tau$ is as above. From \eqref{stimapoint} we can estimate
\begin{align*}
\|\nabla_xG(t,s)f\|_{L^{p_2}(\Rd,\mu_t)}=& \|\nabla_xG(t,t-\tau)G(t-\tau,s)f\|_{L^{p_2}(\Rd,\mu_t)}\nonumber\\
\le & C_2(\tau)\|G(t,t-\tau)|\nabla_xG(t-\tau,s)f|\|_{L^{p_2}(\Rd,\mu_t)}\nonumber\\
\le & C_2(\tau)\|\nabla_xG(t-\tau,s)f\|_{L^{p_1}(\Rd,\mu_{t-\tau})}\nonumber \\
\le & N_{p_1,\omega}C_2(\tau)e^{\omega (t-s-\tau)}\|f\|_{L^{p_1}(\Rd,\mu_s)}\nonumber\\
\le & N_{p_1,\omega}C_2(\tau)e^{-\tau\omega}e^{\omega (t-s)}\|f\|_{L^{p_2}(\Rd,\mu_s)}
\end{align*}
for any $f \in C^1_b(\Rd)$ and some positive constant $N_{p_1,\omega}$, independent of $f$, $s$ and $t$. The density of $C^1_b(\R^d)$ into $L^{p_2}(\Rd,\mu_s)$
allows to extend the previous estimate to any $f\in L^{p_2}(\Rd,\mu_s)$. Again, splitting $\nabla_xG(t,s)f=\nabla_xG(t,t-1)G(t-1,s)f$,
using estimate \eqref{stimapoint} and the contractivity of $G(t-1,s)$ from $L^{p_2}(\Rd,\mu_s)$ to $L^{p_2}(\Rd,\mu_{t-1})$ we cover also the case $t\in (s+1,s+1+\tau)$.
Hence, $\omega\in\mathfrak{B}_{p_2}$.

Viceversa, suppose that $\omega\in \mathfrak{B}_{p_2}$ and  $t\geq s+\tau+1$. Then,
\begin{align*}
\|\nabla_xG(t,s)f\|_{L^{p_1}(\Rd,\mu_t)}\le & \|\nabla_xG(t,s)f\|_{L^{p_2}(\Rd,\mu_t)}\\
= & \|\nabla_xG(t,s+\tau)G(s+\tau,s)f\|_{L^{p_2}(\Rd,\mu_t)}\\
\le & N_{p_2,\omega}e^{\omega(t-s-\tau)}\|G(s+\tau,s)f\|_{L^{p_2}(\Rd,\mu_{s+\tau})}\\
\le & N_{p_2,\omega}e^{-\omega \tau}e^{\omega(t-s)}\|f\|_{L^{p_1}(\Rd,\mu_{s})}
\end{align*}
for  some positive constant $N_{p_2,\omega}$, independent of $f\in L^{p_1}(\Rd,\mu_{s})$, $s$ and $t$.
This is enough to infer that $\omega\in\mathfrak{B}_{p_1}$. We have so proved that $\mathfrak{B}_{p_1}=\mathfrak{B}_{p_2}$ for any $1<p_1<p_2<+\infty$ and this implies that  $\mathfrak{B}_p$ is independent of
$p\in (1,+\infty)$.

(iv) In view of (i)-(iii), to prove that $\mathfrak{A}_p=\mathfrak{B}_p$ for any $p\in (1,+\infty)$, it suffices to show that
$\mathfrak{B}_{2}\subset\mathfrak{A}_{2}$.
Fix $\omega\in \mathfrak{B}_{2}$, $s,t\in I$, with $t-s\ge 1$ and $f\in L^{2}(\R^d,\mu_s)$.
Applying the Poincar\'e inequality (with $\mu_s$ and $f$ replaced by $\mu_t$ and $G(t,s)f$, respectively) and
observing that $\overline{(G(t,s)f)}_t=\overline{f}_s$, we get
\begin{align*}
\|G(t,s)f-\overline f_s\|_{L^{2}(\R^d,\mu_t)}=&\|G(t,s)f- \overline{(G(t,s)f)}_t\|_{L^{2}(\R^d,\mu_t)}\notag\\
\le &2^{-1} C_{2}\|\nabla_x G(t,s)f\|_{L^{2}(\R^d,\mu_t)}\notag\\
\le &2^{-1} C_{2}N_{2,\omega}e^{\omega(t-s)}\|f\|_{L^{2}(\R^d,\mu_s)}.
\end{align*}
This is enough to infer that $\omega\in {\mathfrak A}_{2}$ and we are done.
\end{proof}

\begin{remark}{\rm
\begin{enumerate}[\rm(i)]
\item
Under the assumptions of Theorem \ref{thm_log_sob}, all the conditions in Theorem \ref{Asy} are satisfied and
estimate \eqref{stimapoint} implies that $r_0\in \mathfrak{B}_p$. From the equality $\mathfrak{A}_p=\mathfrak{B}_p$ we deduce that, for any $f \in C_b(\Rd)$ and $p>1$, $\|G(t,s)f-\overline f_s\|_{L^p(\Rd,\mu_t)}$ decays exponentially to zero, as $t\to +\infty$.
\item
The equality $\mathfrak{A}_p=\mathfrak{B}_p$ fails when $p=1$, even in the autonomous case. For instance, in the case of the Ornstein-Uhlenbeck operator $(\A\zeta)(x) := \zeta''(x) - x\zeta'(x)$ we have $d\mu_t=(2\pi)^{-1/2}e^{-x^2/2}dx$ for every $t$, and every $\lambda<0$ is an eigenvalue of the realization of $\A$ in $L^1( \R, \mu)$ as shown in \cite{MPP}. This implies that $\mathfrak{A}_1$ cannot contain negative numbers, so that $\mathfrak{A}_1=[0, +\infty)$. On the other hand, in this case $r_0= -1 \in \mathfrak{B}_1$ by point (i).
\end{enumerate}}
\end{remark}

\medskip

Under the assumptions of Theorem \ref{th-5.2}, Theorem \ref{Asy} provides us with a very strong result, since allows us to prove that
$\|G(t,s)f-\overline{f}_s\|_{L^p(\Rd,\mu_t)}$ decays to zero as $t\to +\infty$ with an exponential rate. On the other hand the assumptions in
Theorem \ref{th-5.2} may sound rather restrictive since the diffusion coefficients are assumed to be bounded and independent of $x$.
As we have already explained this condition is {\it almost} necessary to prove the pointwise estimate \eqref{stimapoint} which is the crucial
tool to prove Theorem \ref{th-5.2}.

The results in \cite{LorLunZam10Asy}, which deals with the case when the coefficients are periodic with respect to the time variable, show that
$\|G(t,s)f-\overline{f}_s\|_{L^p(\Rd,\mu_t)}$ decays to zero even without requiring the validity of \eqref{stimapoint}. Motivated by that result,
in \cite{LorLunSch16Str}, the convergence of $\|G(t,s)f-\overline{f}_s\|_{L^p(\Rd,\mu_t)}$ to zero has been proved also in the nonperiodic setting under the following
conditions on the coefficients of the operator $\A$:
\begin{hypotheses}
\label{hyp-fin}\begin{enumerate}[\rm (i)]
\item
The coefficients $q_{ij}$ and $b_j$ $(i,j=1,\ldots, d)$ belong to $C^{\alpha/2,1+\alpha}_{\rm loc}(I\times\R^d)$;
\item
$q_{ij}\in C_b(I\times B_R)$, $D_hq_{ij}, b_j\in C_b(I;L^p(B_R))$ for any $i,j,h\in\{1,\ldots,d\}$, any $R>0$ and some $p>d+2$;
\item
Hypothesis $\ref{hyp-1}(ii)$ is satisfied;
\item
there exist a positive function $\varphi:\Rd\to\R$, blowing up as $|x|\to +\infty$ and positive constants $a_1$ and $a_2$ such that
$\A\varphi\le a_1-a_2\varphi$ in $\R^{d+1}$;
\item
there exist constants $C_0>0$ and $r_0\in\R$ such that $|\nabla_xq_{ij}|\le C_0\nu$ in $I\times\R^d$ for any $i,j=1,\ldots,d$;
\item
there exists a constant $M>0$ such that either $|q_{ij}(t,x)|\le M(1+|x|)\varphi(x)$ $(i,j=1,\ldots,d)$ and $\langle b(t,x),x\rangle \le C(1+|x|^2)\varphi(x)$
for any $(t,x)\in I\times\Rd$ or $|q_{ij}(t,x)|\le C$ $(i,j=1,\ldots,d)$ for any $(t,x)\in I\times\R^d$.
\end{enumerate}
\end{hypotheses}

The strategy used in \cite{LorLunZam10Asy} is different to that illustrated here (even if it still used the gradient estimate
$|\nabla_xG(t,s)f|^p\le K_0\max\{1,(t-s)^{-p/2}\}G(t,s)|f|^p$)
 and is based on argument from semigroup theory applied to the so-called
evolution semigroup ${\mathscr T}(t)$, which is defined when $I=\R$ by\footnote{Throughout the section, whenever we consider the evolution operator, we assume that the coefficients
are defined in the whole $\R^{d+1}$, in such a way that the assumptions that we use are satisfied with $I=\R$. This is not a restriction since the coefficients can be extended
to $\R^{d+1}$ without adding further conditions.}
$({\mathscr T}(t)f)(s,x)=(G(s,s-t)f(s,\cdot))(x)$ for any $t\ge 0$, $(s,x)\in\R^{d+1}$ and $f\in C_b(\R^{d+1})$.
This semigroup can be extended to the $L^p$-spaces related to the unique Borel measure $\mu$ such that
\begin{eqnarray*}
\mu(A\times B)=\int_A\mu_t(B)dt
\end{eqnarray*}
for any pair of Borel sets $A\subset\R$ and $B\subset\Rd$. This follows from the invariance of the evolution system of measures $\{\mu_t: t\in I\}$ which implies that
\begin{eqnarray*}
\int_{\R^{d+1}}{\mathscr T}(t)fd\mu=\int_{\Rd}fd\mu,\qquad t>0,\;\,f\in C^{\infty}_c(\R^{d+1}).
\end{eqnarray*}
Note that $\mu$ is not a probability measure since $\mu(\R^{d+1})=+\infty$.
The arguments used in \cite{LorLunZam10Asy} relies on the fact that, under quite general assumptions on the coefficients of the operator $\A$,
\begin{equation}
\lim_{t\to +\infty}\|\nabla_x{\mathscr T}(t)f\|_{L^p(\Rd,\mu)}=0,\qquad\;\, f\in L^p(\R^{d+1},\mu).
\label{sem}
\end{equation}
This result is proved using only tools from semigroup theory (for the case $p=2$) and an interpolation argument in the case $p\neq 2$. If one applies \eqref{sem}
to the functions $\vartheta_m f$, where $f\in C^{\infty}_c(\Rd)$ and $\vartheta_m\in C^{\infty}_c(\R)$ satisfies the condition
$\chi_{[-m,m]}\le \vartheta_m\le\chi_{[-m-1,m+1]}$ for any $m\in\N$, one easily obtains that there exists a sequence $(t_n)$ such that
\begin{equation}
\lim_{n\to +\infty}\int_{\Rd}\rho(s+t_n,\cdot)|\nabla_xG(s+t_n,s)f|^pdx=0
\label{sem-1}
\end{equation}
for any $s\in\Rd\setminus N$, where $N$ is a negligible set with respect to the Lebesgue measure and $\rho$ is the continuous function in Remark \ref{rem-luca}, which is the density of $\mu$ with respect to the Lebesgue measure.
In the periodic case it is straightforward to infer that the sequence $(\rho(s_+t_n,\cdot))$ is bounded from below by a positive constant in any ball of $\Rd$. In the nonperiodic case, the proof
of this property demands somehow more delicate arguments and the use of Hypothesis \ref{hyp-fin}(ii). In any case, from \eqref{sem-1} we conclude that
the sequence $(|\nabla_xG(s+t_n,s)f|)$ vanishes in $L^p(B_k)$ for any $k\in\N$ as $n\to +\infty$. Since the sequence $(\|G(s+t_n,s)f\|_{L^p(\Rd,\mu_{s+t_n})})$ is bounded and $\rho$ is continuous in
$\R^{d+1}$, the sequence $(G(s+t_n,s)f)$ is bounded in $W^{1,p}(B_k)$ for any $k\in\N$. By the Rellich-Kondrachov theorem it follows that, up to a subsequence,
$G(s+t_n,s)f$ converges in $W^{1,p}_{\rm loc}(\Rd)$ to a constant function $g(s)$ and the convergence is also local uniform if we take $p>d$. To identify
$g(s)$ with $f_s$, it suffices to use the invariance property of $\{\mu_t: t\in I\}$ to write
\begin{eqnarray*}
f_s-g(s)=\int_{\Rd}(f-g(s))d\mu_s=\int_{\Rd}G(s+t_n,s)(f-g(s))d\mu_{s+t_n}
\end{eqnarray*}
and let $n$ tend to $+\infty$. Since the function $t\mapsto \|G(t,s)f-\overline f_s\|_{L^p(\Rd,\mu_s)}$ is decreasing in $(s,+\infty)$, from the above result we conclude that
$\|G(t,s)f-\overline f_s\|_{L^p(\Rd,\mu_s)}$ tends to $0$ as $t\to +\infty$ for any $s\not\in N$ and $f\in C^{\infty}_c(\Rd)$.
By density, we can replace $C^{\infty}_c(\Rd)$ with $L^p(\Rd,\mu_s)$ and the evolution law allow to remove the condition $s\notin\N$.

\medskip

To conclude this section, we stress that in the limit $\lim_{t\to +\infty}\|G(t,s)f-\overline f_s\|_{L^p(\Rd,\mu_s)}=0$ also the $L^p$-space varies with $t$.
It thus makes sense to (i) study the behaviour as $t\to +\infty$ of the measures $\mu_t$ determining the point limit
(ii) establish whether the convergence of $G(t,s)f$ to $\overline f_s$ may be guaranteed also in some {\it fixed} $L^p$-space.
In the periodic case (i) it is easy since
the function $t\mapsto \mu_t$ is periodic.  In the general case, the previous points have been addressed in
\cite{AngLorLun,LorLunSch16Str}. Here, we state the (more general) result proved in \cite{LorLunSch16Str}.
Under Hypotheses \ref{hyp-fin} and assuming that the coefficients $q_{ij}$ and $b_j$ $(i,j=1,\ldots,d)$ belong to
$C^{\alpha/2,\alpha}_b([s_0,+\infty)\times B_R)$ for any $R>0$ and some $s_0\in I$, and they converge pointwise in $\Rd$ as $t\to +\infty$,
in \cite[Proposition 4.3]{LorLunSch16Str} it has been proved that the density of $\mu_t$ converges to a function $\rho_{\infty}$
locally uniformly in $\Rd$ and in $L^1(\Rd)$. $\rho_{\infty}$ is the density (with respect to the Lebesgue measure) of
the invariant measure $\mu_{\infty}$ of the semigroup associated with the elliptic operator whose coefficients are the limit as $t\to +\infty$ of the coefficients
of the operator $\A(t)$. This result has been used to answer point (ii). More precisely, in \cite[Theorem 4.4]{LorLunSch16Str} it has been proved that for any $f\in C_b(\Rd)$, $G(t,s)f$ converges to $\overline f_s$, as $t\to +\infty$, in $L^p(\Rd,\mu_{\infty})$ for any $p\in [1,+\infty)$ and any $s\in I$.

%---------------------------------------------- EQUATIONS:

%--------------------------------------------acknowledgments:
%\vspace{0.5cm} \indent {\it
%A\,c\,k\,n\,o\,w\,l\,e\,d\,g\,m\,e\,n\,t\,s.\;} The authors
%acknowledge support by the \ldots.

%%%----------------------------------------------%%%
%%%                BIBLIOGRAPHY                  %%%
%%%  REFERENCES SHOULD BE LISTED ALPHABETICALLY  %%%
%%%----------------------------------------------%%%
\bigskip
\begin{center}

\end{center}

\end{document}